\documentclass[11pt]{scrartcl}

\newif\ifcomment
\commenttrue

\usepackage[utf8]{inputenc}
\usepackage[T1]{fontenc}
\usepackage{lmodern}

\usepackage{ifthen}
\usepackage{amssymb}
\usepackage{microtype}
\usepackage{amsmath}
\usepackage{amssymb}
\usepackage{amsfonts}
\usepackage{mathtools}
\usepackage{xcolor}
\usepackage{xstring}

 \usepackage[hyphens]{url}
 \usepackage{amsthm}
 \usepackage{hyperref}
 \usepackage{cleveref}
 \usepackage{tikz}
 \usepackage{subcaption}

 \usepackage[inline]{enumitem}
 \usepackage{titling}
 \usepackage{xspace}
 \usepackage{dsfont}

 \setlist[itemize]{topsep=0pt,partopsep=0pt,itemsep=0pt,parsep=0pt}
 \setlist[itemize,1]{label=-}
 \setlist[itemize,2]{label=---}
 \setlist[itemize,3]{label=*}
 \setlist[enumerate]{topsep=0pt,partopsep=0pt,itemsep=0pt,parsep=0pt}
 \setlist[enumerate,1]{label=\roman*)}
 \setlist[enumerate,2]{label=\alph*)}
 \setlist[enumerate,3]{label=\arabic*)}
 
\usetikzlibrary{shapes.geometric,hobby,calc} 
\usetikzlibrary{decorations}
\usetikzlibrary{decorations.pathmorphing}
\usetikzlibrary{decorations.text}
\usetikzlibrary{shapes.misc}
\usetikzlibrary{decorations,shapes,snakes}

\usepackage{macros}

\usepackage{stackengine}
\stackMath

\hypersetup{
	colorlinks=true,
	linkcolor=AO!70!black,
	citecolor=AO!70!black,
	urlcolor=AppleGreen!70!black,
	bookmarksopen=true,
	bookmarksnumbered,
	bookmarksopenlevel=2,
	bookmarksdepth=3
}

\title{A Note on Directed Treewidth}
\predate{}
\date{}
\postdate{}

\preauthor{}
\DeclareRobustCommand{\authorthing}{
\begin{center}
\begin{tabular}{c}
Sebastian Wiederrecht\\
\emph{sebastian.wiederrecht@tu-berlin.de}\\
\noalign{\smallskip}
{Technische Universität Berlin}
\end{tabular}
\end{center}}
\author{\authorthing}
\postauthor{}

\begin{document}
\maketitle

\begin{abstract}
	We characterise digraphs of directed treewidth one in terms of forbidden butterfly minors.
	Moreover, we show that there is a linear relation between the hypertree-width of the dual of the cycle hypergraph of $D$, i.\@ e.\@ the hypergraph with vertices $\Fkt{V}{D}$ where every hyperedge corresponds to a directed cycle in $D$, and the directed treewidth of $D$.
	Based on this we show that a digraph has directed treewidth one if and only if its cycle hypergraph is a hypertree.
	
	\noindent \textbf{Keywords:} directed treewidth, butterfly minor, hypertree-width
\end{abstract}

\section{Introduction}

The notion of \emph{treewidth}, introduced in its popular form by Robertson and Seymour \cite{graphminorproject}, can be seen as an invariant for (undirected) graphs that measures how far away from a tree a given graph $G$ is.
Treewidth is closed under minors and thus classes of bounded treewidth can be described by finite sets of so called \emph{forbidden minors}.
It is well known that the graphs of treewidth at most one are exactly those excluding $C_3$ as a minor.
So a connected graph has treewidth one if and only if it is a tree.
Similarly other classes of small treewidth can be characterised by (finite) families of forbidden minors.

In the more general setting of directed graphs, it is not immediately clear hot to generalise the notion of treewidth and several attempts haven been made.
From a structural point of view the most popular generalisation of treewidth to digraphs is the \emph{directed treewidth} by Johnson et al.\@ \cite{johnson2001directed}.
Sadly, directed treewidth is not (tightly) closed under reasonable versions of minors for directed graphs \cite{adler2007directed} and so one probably cannot expect forbidden minor characterisations for graphs of directed treewidth at most $k$.
With some exceptions at least.
The examples showing that directed treewidth might increase after contracting certain edges already require digraphs of directed treewidth $3$ and no example with smaller directed treewidth seems to be known.
Hence especially for the case $k=1$ one could still hope for a characterisation.
Please note that the commonly used notion of \emph{butterfly minors} does not yield a well quasi ordering for all digraphs.
Thus minimal families of forbidden minors might be infinite.

The DAGs, \emph{directed acyclic graphs}, all have directed treewidth $0$.
DAGs can be arbitrary dense and since they do not have any directed cycles they behave like stable sets in the undirected case with regards to (directed) treewidth.
This becomes apparent if one considers the \emph{cops and robber} game for directed treewidth.
Here the robber is only allowed to move within strong components and thus cannot move at all in a DAG.
So if one were to seek a directed analogue of undirected trees the digraphs of directed treewidth one could be one way to find one.

\paragraph{Contribution}

We give two answers to the question of digraphs of directed treewidth one.
That is, we characterise the directed treewidth one digraphs by a minimal, although infinite, family of forbidden butterfly minors and we also characterise these digraphs by the structure of their \emph{cycle hypergraphs}.
The cycle hypergraph $\CycleHypergraph{D}$ of a digraph $D$ is the hypergraph obtained from $D$ by taking $\Fkt{V}{D}$ as the vertex set and the vertex sets of all directed cycles in $D$ as the set of hyperedges.
We show that a digraph $D$ has directed treewidth one if and only if $\CycleHypergraph{D}$ is a hypertree.
Previously, besides some small classes of digraphs \cite{gurski2019forbidden}, almost nothing was known on digraphs of directed treewidth one.
As another novelty we adapt techniques and tools from structural matching theory directly for the setting of digraphs, without needing to invoke the matching theoretic setting itself.

The second characterisation of directed-treewidth-one digraphs is part of a bigger picture.
We show that the \emph{hypertree-width} of the \emph{dual} cycle hypergraph $\Dual{\CycleHypergraph{D}}$ is closely linked to the directed treewidth of $D$.
We believe that this makes a strong point for the directed treewidth as a structural digraph width measure.

\section{Preliminaries}

Throughout this paper we will stick to the following convention, by $G$ we denote \emph{undirected graphs}, by $D$ we denote \emph{directed graphs} and by $H$ we denote \emph{hypergraphs} with the exception of the cycle hypergraph.

We only consider simple graphs, so we do not allow loops or multiple edges.
If $G$ is a graph and $u,v\in\Fkt{V}{G}$ we denote the edge $\Set{u,v}$ by $uv$ or $vu$ which can be understood as the same edge.
A path in $G$ is denoted by only giving its vertices, so a \emph{path} of length $\ell$ is a sequence of pairwise distinct vertices $P=\Brace{v_0,v_1,\dots,v_{\ell}}$ such that $v_iv_{i+1}\in\Fkt{E}{G}$ for all $i\in\Set{0,\dots,\ell-1}$.
We say that $v_0$ and $v_{\ell}$ are the \emph{endpoints} of $P$ and by a slight abuse of notation we also identify the graph $\Brace{\Set{v_0,\dots,v_{\ell}},\Set{v_0v_1,\dots,v_{\ell-1}v_{\ell}}}$ with $P$.
For a more in-depth introduction and more definitions we refer to \cite{DBLP:books/daglib/0030488}.

\begin{definition}
Let $G$ be a graph, a \emph{tree decomposition} for $G$ is a tuple $\Brace{T,\beta}$ where $T$ is a tree and $\beta\colon\Fkt{V}{T}\rightarrow 2^{V(G)}$ is a function, called the \emph{bags} of the decomposition, satisfying the following properties:
\begin{enumerate}
	\item $\bigcup_{t\in V(T)}\Fkt{\beta}{t}=\Fkt{V}{G}$,
	
	\item for every $e\in\Fkt{E}{G}$ there exists some $t\in\Fkt{V}{T}$ such that $e\subseteq\Fkt{\beta}{t}$, and
	
	\item for every $v\in\Fkt{V}{G}$ the set $\CondSet{t\in\Fkt{V}{T}}{v\in\Fkt{\beta}{t}}$ induces a subtree of $T$.
\end{enumerate}	
The \emph{width} of $\Brace{T,\beta}$ is defined as
\begin{align*}
\Width{\Brace{T,\beta}}\coloneqq \max_{t\in V(T)}\Abs{\Fkt{\beta}{t}}-1.
\end{align*}
The \emph{\treewidth} of $G$, denoted by $\tw{G}$, is defined as the minimum width over all tree decompositions for $G$.
\end{definition}

The upcoming two subsections will introduce and give short overviews on possible generalisations of \treewidth appropriate for the settings of digraphs an hypergraphs.

\subsection{Digraphs and Directed Treewidth}

A \emph{directed graph}, or \emph{digraph for short}, is a tuple $D=\Brace{V,E}$ where $V$, or $\Fkt{V}{D}$, is the \emph{vertex set}of $D$ while $\Fkt{E}{D}=E\subseteq V\times V$ is the \emph{edge set}, sometimes called \emph{arc set}, of $D$.
All digraphs will be \emph{simple}, so we do not allow multiple edges or loops, although we allow 
\emph{digons}, i.\@ e.\@ $D$ may contain both edges $\Brace{u,v}$ and $\Brace{v,u}$ at the same time.
For an edge $e=\Brace{u,v}\in\Fkt{E}{D}$ we call $u$ the \emph{tail} of $e$ while $v$ is its \emph{head}.

Given a vertex $v\in\Fkt{V}{D}$ we say that an edge $e$ is \emph{out-going} from $v$ if $v$ is the tail of $e$.
In case $v$ is the head of $e$ we say that $e$ is \emph{incoming} to $v$.
In either way $e$ is \emph{incident} with $v$, we write $v\sim e$.
The \emph{out-neighbourhood} of $v$, denoted by $\OutN{D}{v}$, is the set of vertices $u$ for which the edge $\Brace{v,u}$ exists.
Similarly the \emph{in-neighbourhood} of $v$, denoted by $\InN{D}{v}$, is the set of vertices $u$ for which the edge $\Brace{u,v}$ exists.
We adapt our notation for paths in undirected graphs to digraphs, so a \emph{directed path} of \emph{length} $\ell$ is a sequence $P=\Brace{v_0,v_1,\dots,v_{\ell}}$ of pairwise distinct vertices such that $\Brace{v_,v_{i+1}}\in\Fkt{E}{D}$ for all $i\in\Set{0,\dots,\ell-1}$.
We say that $P$ \emph{goes from $v_0$ to $v_{\ell}$}.
Sometimes need a relaxed version of directed paths is needed, namely \emph{directed walks}, in which the vertices do not have to be pairwise distinct.
Again we identify the sequence of vertices and the actual directed graph representing the directed path.
A \emph{directed cycle} of \emph{length} $\ell$ is a sequence $C=\Brace{v_1,\dots,v_{\ell},v_1}$ where the $v_i$, $v_j$ are distinct for $i,j\in\Set{1,\dots,\ell}$, $i\neq j$, and $\Brace{v_i,v_{i+1}}\in\Fkt{E}{D}$ for all $i\in\Set{1,\dots,\ell-1}$ as well as $\Brace{v_{\ell},v_1}\in\Fkt{E}{D}$.

For two vertices $u,v\in\Fkt{V}{D}$ we write $\Reaches{u}{D}{v}$ if there exists a directed path that goes from $u$ to $v$ in $D$, note that $\Reaches{u}{D}{u}$ holds trivially for all $u$.
A digraph $D$ is called \emph{strongly connected} if $\Reaches{u}{D}{v}$ and $\Reaches{v}{D}{u}$ for all pairs of vertices $u,v\in\Fkt{V}{D}$.
We denote by $\Below{D}{v}$ the set of all vertices $u$ with $\Reaches{v}{D}{u}$.
A maximal strongly connected subgraph of $D$ is a \emph{strong component} of $D$.

Let $G$ be an undirected graph, the \emph{bidirection} of $G$ is the following digraph
\begin{align*}
\Bidirected{G}\coloneqq\Brace{\Fkt{V}{G},\CondSet{\Brace{u,v},\Brace{v,u}}{uv\in\Fkt{E}{G}}}.
\end{align*}
A \emph{bicycle} is the bidirection of an undirected cycle $C$ and its \emph{length} is the length of $C$.
An \emph{orientation} of $G$ is any digraph that can be obtained from $G$ by replacing every edge $uv\in\Fkt{E}{G}$ by exactly one of the to possible directed edges $\Brace{u,v}$ and $\Brace{v,u}$.

An \emph{arborescence} is an orientation $T$ of an undirected tree such that there exists a unique vertex $r\in\Fkt{V}{T}$ with $\Below{T}{r}=\Fkt{V}{T}$.
We call $r$ the \emph{root} of $T$.
If $T$ is an arborescence and $t\in\Fkt{V}{T}$ we denote by $T_t$ the subgraph of $T$ induced by $\Below{T}{t}$ called the \emph{subtree of $T$ rooted at $t$}.

\begin{definition}
Let $D$ be a digraph, a \emph{directed tree decomposition} for $D$ is a tuple $\Brace{T,\beta,\gamma}$ where $T$ is an arborescence, $\beta\colon\Fkt{V}{T}\rightarrow 2^{V(D)}$ is a function that partitions $\Fkt{V}{D}$, called the \emph{bags}, and $\gamma\colon\Fkt{E}{T}\rightarrow 2^{V(D)}$ is a function, called the \emph{guards}, satisfying the following requirement:
\begin{enumerate}
	\item[] For every $\Brace{d,t}\in\Fkt{E}{T}$, $\Fkt{\gamma}{\Brace{d,t}}$ is a hitting set for all directed walks in $D$ that start and end in $\Fkt{\beta}{T_t}\coloneqq\bigcup_{t'\in V(T_t)}\Fkt{\beta}{t'}$ and contain a vertex that does not belong to $\Fkt{\beta}{T_t}$.
\end{enumerate}
For every $t\in\Fkt{V}{T}$ let $\Fkt{\Gamma}{t}\coloneqq\Fkt{\beta}{t}\cup\bigcup_{t\sim e}\Fkt{\gamma}{e}$.
The \emph{width} of $\Brace{T,\beta,\gamma}$ is defined as
\begin{align*}
\Width{\Brace{T,\beta,\gamma}}\coloneqq\max_{t\in V(T)}\Abs{\Fkt{\Gamma}{t}}-1.
\end{align*}
The \emph{\dtwText} of $D$, denoted by $\dtw{D}$, is the minimum width over all directed tree decompositions for $D$.
\end{definition}

A \emph{complexity measure} for a class of combinatorial object $\mathcal{O}$ is a function $\nu$ that assigns every $O\in\mathcal{O}$ a non-negative integer.
We say that two complexity measures $\nu$ and $\mu$ for $\mathcal{O}$ are \emph{equivalent} if there exist computable and monotone functions $f$ and $g$ such that for all $O\in\mathcal{O}$ we have
\begin{align*}
\Fkt{f}{\Fkt{\mu}{O}}\leq\Fkt{\nu}{O}\leq\Fkt{g}{\Fkt{\mu}{O}}.
\end{align*}

One of the arguably most popular characterisations of undirected \treewidth is in terms of a so called \emph{cops and robber} game.
A similar game can be defined for \dtwText and its equivalence to \dtwText was among the very first results regarding this width measure \cite{johnson2001directed}.
Tightly bound to the directed cops and robber game is the notion of havens, which gives a more structural description of winning strategies of the robber player.

The cops and robber game for directed treewidth on a digraph $D$ is played as follows.
In the beginning the cops choose a position $X_0\subseteq\Fkt{V}{D}$ and the robber chooses a strong component $R_0$ of $D-X_0$, this defines the initial position $\Brace{X_0,R_0}$.
If $\Brace{X_i,R_i}$ is a position in the game, the cops may announce a new position $X_{i+1}\subseteq\Fkt{V}{D}$ and the robber chooses a new strong component $R_{i+1}$ of $D-x_{i+1}$ such that $X_i$ and $X_{i+1}$ belong to the same strong component of $D-\Brace{X_i\cap X_{i+1}}$.
If the robber cannot choose a non-empty component, she looses the game.
Otherwise, if the cops cannot catch the robber in a finite number of turns, the robber wins.
The \emph{directed cop number}, denoted by $\dcn{D}$, is the smallest number $k$ such that the cops win and $\Abs{X_i}\leq k$ for all $i$.
For a formal definition the reader is referred to \cite{DBLP:books/sp/18/KreutzerK18}.

\begin{definition}
Let $D$ be a digraph.
A \emph{haven} of order $k$ is a function $h\colon\Choose{\Fkt{V}{D}}{<k}\rightarrow 2^{V(G)}$ assigning to every set $X$ of fewer than $k$ vertices a strong component of $D-X$ such that if $Y\subseteq X\in\Choose{\Fkt{V}{D}}{<k}$, then $\Fkt{h}{X}\subseteq\Fkt{h}{Y}$.
\end{definition}

\begin{lemma}[\cite{johnson2001directed}]\label{lemma:dcopsandkhavens}
Let $D$ be a digraph.
If $D$ contains a haven of order $k$, then the robber has a winning strategy against $k-1$ cops in the directed cops and robber game and $\dcn{D}-1\leq\dtw{D}$.
\end{lemma}

In particular this means that, if $D$ contains a haven of order $k+1$, then the \dtwText of $D$ is at least $k$.
The notion of havens can be used to deduce that the directed cop number and \dtwText are in fact equivalent complexity measures for digraphs.

\begin{theorem}[\cite{johnson2001directed}]\label{thm:dhavensvsdtw}
Let $D$ be a digraph, then either $D$ contains a haven of order $k$, or has \dtwText at most $3k-2$.
\end{theorem}

So if $k$ cops have a winning strategy in the directed cops and robber game, $D$ cannot contain a haven of order $k+1$ implying that the \dtwText of $D$ is at most $3k+1$.

Another well known obstruction to \treewidth is the notion of highly linked sets.
Roughly speaking, a highly linked set cannot be split in a balanced way by deleting only a small number of vertices.
A similar idea can be applied to digraphs.

\begin{definition}
Let $D$ be a digraph.
A set $W\subseteq\Fkt{V}{D}$ is called \emph{$k$-linked} if for every set $S\subseteq\Fkt{V}{D}$ with $\Abs{S}\leq k$ there exists a strong component $K$ of $D-S$ such that $\Abs{\Fkt{V}{K}\cap W}>\frac{\Abs{W}}{2}$.	
\end{definition}

For $k$-linked sets one can obtain a duality theorem for \dtwText strikingly similar to the one for havens.
While the theorem should be contributed to Johnson et al.\@, details of the proof for the following theorem can be found in \cite{reed1999introducing} as well.

\begin{theorem}[\cite{johnson2001directed}]\label{thm:dklinkedvsdtw}
Let $D$ be a digraph.
If $\dtw{D}\leq k$, $D$ does not contain a $\Brace{k+1}$-linked set.
Moreover, $D$ either contains a $k$-linked set or has \dtwText at most $3k-2$.
\end{theorem}

For the sake of completeness we have to mention the Directed Grid Theorem, which, although not as precise as the duality theorems we have seen so far, still might be one of the deepest and most important results regarding \dtwText.
An edge $\Brace{u,v}$ is \emph{contractible} if it is the only out-going edge of $u$, or the only incoming edge at $v$.
The \emph{butterfly contraction} of $\Brace{u,v}$ is the operation of identifying $u$ and $v$ into a single vertex and deleting all resulting parallel edges and all loops.
A digraph $D'$ is a \emph{butterfly minor} of $D$ if it can be obtained from a subgraph of $D$ by a series of butterfly contractions.

Although \dtwText was shown to not be monotone under butterfly minors in \cite{adler2007directed}, it was shown in \cite{hatzel2019cyclewidth} that there still is a function bounding the \dtwText of a butterfly minor $D'$ of $D$ in $\dtw{D}$, a result that is also implied by the Directed Grid Theorem itself.
A \emph{cylindrical grid} of order $k$ consists of $k$ concentric directed cycles of length $2k$ and $2k$ vertex disjoint paths connecting the cycles in alternating directions.

\begin{theorem}[\cite{kawarabayashi2015directed}]\label{thm:dirgridthm}
There exists a function $f$ such that every digraph $D$ either has \dtwText at most $\Fkt{f}{k}$, or contains the cylindrical grid of order $k$ as a butterfly minor.
\end{theorem}

\subsection{Hypergraphs and Hypertree-Width}

A \emph{hypergraph} is a tuple $H=\Brace{V,E}$ where $\Fkt{V}{H}=V$ is the \emph{vertex set} and $\Fkt{E}{H}=E\subseteq 2^{V(H)}$ is the (multi-)set of \emph{(hyper-)edges}.
A vertex of $H$ is \emph{isolated} if there is no edge containing it.
We do not allow isolated vertices, however we will allow multiple edges.

For a hypergraph $H$ and a set $X\subseteq\Fkt{V}{H}$, the \emph{subhypergraph of $H$ reduced to $X$} is the hypergraph $\InducedSubgraph{H}{X}\coloneqq\Brace{X,\CondSet{e\cap X\neq\emptyset}{e\in\Fkt{E}{H}}}$, after the removal of possible isolated vertices.
The \emph{subhypergraph of $H$ reduced by $X$} is then defined as $H\setminus X\coloneqq\InducedSubgraph{H}{\Fkt{V}{H}\setminus X}$.
The \emph{$2$-section}, or \emph{primal graph} of $H$ is the graph
\begin{align*}
	\underline{H}\coloneqq\Brace{\Fkt{V}{H},\CondSet{uv}{u,v\in\Fkt{V}{H}~\text{and}~\Set{u,v}\subseteq e~\text{for some}~e\in\Fkt{E}{H}}}.
\end{align*}
The \emph{dual hypergraph} of $H$ is the hypergraph
\begin{align*}
\Dual{H}\coloneqq\Brace{\Fkt{E}{H},\CondSet{e_v\coloneqq\CondSet{e\in\Fkt{E}{H}}{v\in e}}{v\in\Fkt{V}{H}}}.
\end{align*}
Finally the \emph{linegraph} of $H$ is the graph
\begin{align*}
\LineGraph{H}\coloneqq\Brace{\Fkt{E}{H},\CondSet{ef}{e,f\in\Fkt{E}{H}~\text{and}~e\cap f\neq\emptyset}}.
\end{align*}
Please note that $\LineGraph{H}=\underline{\Dual{H}}$.

The \emph{cycle-hypergraph} of a directed graph $D$ is defined as follows:
\begin{align*}
\CycleHypergraph{D}\coloneqq\Brace{\Fkt{V}{D},\CondSet{\Fkt{V}{C}}{C\subseteq D~\text{is a directed cycle}}}.
\end{align*}
Please note that we might have different directed cycles in $D$ on the same vertex set and so, since we allow multiple edges, we have different hyperedges for those directed cycles.

A hypergraph $H$ is connected if $\underline{H}$ is \emph{connected} and a set $X\subseteq\Fkt{V}{H}$ is \emph{connected} if the reduced hypergraph $\InducedSubgraph{H}{X}$ is connected.
A \emph{connected component} is a maximal connected set of $\Fkt{V}{H}$.
We define paths analogous to paths in undirected graphs, indeed a sequence $\Brace{v_0,v_1,\dots,v_{\ell}}$ of vertices of $H$ is a path if it is a path in $\underline{H}$.
A \emph{tree decomposition} for a hypergraph $H$ is a tree decomposition of $\underline{H}$.

\begin{definition}
Let $H$ be a hypergraph.
A \emph{generalised hypertree decomposition} for $H$ is a tripe $\Brace{T,\beta,\gamma}$ where $\Brace{T,\beta}$ is a tree decomposition for $H$ and $\gamma\colon\Fkt{V}{T}\rightarrow 2^{E(H)}$ such that $\Fkt{\beta}{t}\subseteq\bigcup\Fkt{\gamma}{t}\coloneqq\bigcup_{e\in \gamma(t)}e$ for all $t\in\Fkt{V}{T}$.
We call $\Fkt{\gamma}{t}$ the \emph{guards} at $t$.
The \emph{width} of a generalised hypertree decomposition is defined as
\begin{align*}
\Width{\Brace{T,\beta,\gamma}}\coloneqq\max_{t\in V(T)}\Abs{\Fkt{\gamma}{t}}.
\end{align*}
The \emph{generalised hypertree-width} of $H$, denoted by $\ghw{H}$, is defined as the minimum width over all generalised hypertree decompositions for $H$.
\end{definition}

\begin{definition}
Let $H$ be a hypergraph.
A \emph{hypertree decomposition} for $H$ is a generalised hypertree decomposition $\Brace{T,\beta,\gamma}$ where $T$ is an arborescence that satisfies the following additional condition:
\begin{enumerate}
	\item[] $\Brace{\bigcup\Fkt{\gamma}{t}}\cap\bigcup_{t'\in V(T_t)}\Fkt{\beta}{t'}\subseteq\Fkt{\beta}{t}$ for all $t\in\Fkt{V}{T}$.
\end{enumerate}
The \emph{hypertree-width} of $H$, denoted by $\hw{H}$, is the minimum width over all hypertree decompositions for $H$.
\end{definition}

\begin{lemma}[\cite{gottlob2000comparison,gottlob2002hypertree}]
Let $H$ be a hypergraph, then $\ghw{H}\leq\hw{H}$.
\end{lemma}

Similar to (undirected) \treewidth and directed treewidth there are several structural obstructions to hypertree-width.
For more details we refer to \cite{adler2007hypertree}.

\begin{definition}
Let $H$ be a hypergraph, a set $W\subseteq\Fkt{E}{H}$ is \emph{$k$-hyperlinked} if for every $S\subseteq\Fkt{E}{H}$ with $\Abs{S}<k$, $H\setminus\bigcup S$ has a component $K$ such that
\begin{align*}
\Abs{\CondSet{e\in W}{e\cap\Fkt{V}{K}\neq\emptyset}}>\frac{\Abs{W}}{2}.
\end{align*}
\end{definition}

\begin{theorem}[\cite{adler2007hypertree}]\label{thm:hlinkedvhw}
Let $H$ be a hypergraph.
If $H$ contains a $k$-hyperlinked set, then $k\leq\ghw{H}$.
Moreover, either $H$ contains a $k$-hyperlinked set, or $\hw{G}\leq 3k+1$.
\end{theorem}

Similar to \treewidth-style decompositions, there are width parameters based on so called \emph{branch decompositions} for many combinatorial objects.
In general a branch decomposition consists of a subcubic tree, i.\@ e.\@ trees with maximum degree at most three, a bijection from the set of leaves $\LineGraph{T}$ some ground set of the object in question, and a way to evaluate the bipartitions of said ground set induced by the bijection and the edges ot $T$.

For a tree $T$ and an edge $t_1t_2\in\Fkt{E}{T}$ we denote by $T_{t_i}$ the unique component of $T-t_1t_2$ containing $t_i$.
If $\delta\colon\LineGraph{T}\rightarrow X$ is a bijection between the leaves of $T$ and some set $X$, we set $\Fkt{\delta}{T_{t_i}}\coloneqq\CondSet{\Fkt{\delta}{\ell}}{\ell\in\LineGraph{T}\cap\Fkt{V}{T_{t_i}}}$.

\begin{definition}
Let $H$ be a hypergraph.
A \emph{hyperbranch decomposition} for $H$ is a tuple $\Brace{T,\delta}$ where $T$ is a subcubic tree and $\delta\colon\LineGraph{T}\rightarrow\Fkt{E}{H}$ is a bijection.
The \emph{thickness} of an edge $t_1t_2\in\Fkt{E}{T}$ is defined as follows:
\begin{align*}
\Thickness{t_1t_2}\coloneqq \min \Set{\Abs{S}\left|~S\subseteq\Fkt{E}{H}~\text{such that}~\bigcup\Fkt{\delta}{t_1}\cap\bigcup\Fkt{\delta}{T_{t_2}}\subseteq\bigcup S\right.}.
\end{align*}
The \emph{width} of a hyperbranch decomposition $\Brace{T,\delta}$ is defined as
\begin{align*}
\Width{\Brace{T,\delta}}\coloneqq\max_{e\in E(T)}\Thickness{e}.
\end{align*}
The \emph{hyperbranch-width} of $H$, denoted by $\hbw{H}$, is the minimum width over all hyperbranch decompositions for $H$.
\end{definition}

\begin{theorem}[\cite{adler2007hypertree}]\label{thm:hbwvshw}
Let $H$ be a hypergraph, then $\hbw{H}\leq\ghw{H}\leq\hw{H}\leq9\hbw{H}+1$.
\end{theorem}

\section{Directed Treewidth hand Hypertree-Width}

This section is dedicated to establish a connection between directed treewidth and \hwText.

Let $D$ be a digraph together with a directed tree decomposition $\Brace{T_D,\beta_D,\gamma_D}$ as well as $H$ a hypergraph together with a hypertree decomposition $\Brace{T_H,\beta_H,\gamma_H}$.
Note that in both decomposition concepts we need bags to contain the vertices of the underlying object - be it a digraph or a hypergraph - and guards to cover or separate the vertices contained in rooted subtrees.
However, the guards in the case of digraphs are sets of vertices designed to guard certain directed walks, while the guards in the case of hypergraphs are edges chosen to cover all vertices within the bag.
Taking a step back from the actual directed tree decomposition and considering the equivalent cops and robber game it becomes apparent that the purpose of the guards - or cops in this case now - is in fact not to prevent the robber from traversing along certain directed walks, but to destroy strong connectivity.
If one wanted to have a set $X$ of vertices in a digraph to isolate - in the sense of strong connectivity - from the rest of a digraph by deleting some vertices, the easiest way to do this would be to find a hitting set for all directed cycles in $D$ that have a vertex in $X$ and one in $\Complement{X}$.
If we were to consider the cycle-hypergraph $\CycleHypergraph{D}$ associated with $D$ we would say, that we were looking for a hitting set for all edges of $\CycleHypergraph{D}$ in the cut $\Cut{\CycleHypergraph{D}}{X}\coloneqq\CondSet{e\in\Fkt{E}{\CycleHypergraph{D}}}{e\cap X\neq\emptyset~\text{and}~e\cap\Complement{X}\neq\emptyset}$.

We are now going to translate this situation into the language of the hypergraphic dual $\Dual{\CycleHypergraph{D}}$.
In the dual hypergraph every vertex $v$ becomes a (hyper-)edge, say $e_v$, and every edge $C\in\Fkt{E}{\CycleHypergraph{D}}$ becomes a vertex, say $v_C$.
Observe that $C\in\Cut{\CycleHypergraph{D}}{X}$ holds if and only if $v_C\in
\Brace{\bigcup_{v\in X\phantom{\Complement{X}\!\!\!\!\!\!}}e_v}\cap\Brace{\bigcup_{v\in\Complement{X}}e_v}$.
Similarly, $S\subseteq\Fkt{V}{\CycleHypergraph{D}}$ is a hitting set for $\Cut{\CycleHypergraph{D}}{X}$ if and only if for every $v_C\in
\Brace{\bigcup_{v\in X\phantom{\Complement{X}\!\!\!\!\!\!}}e_v}\cap\Brace{\bigcup_{v\in\Complement{X}}e_v}$ there exists some $x\in S$ such that $v_C\in e_x$.

This section is organised as follows:
We first define a notion of directed branchwidth and show its equivalence to directed treewidth.
Then we give a more formal proof of the observations from the discussion above and show that the directed branchwidth of a digraph $D$ coincides with the hyperbranch width of its dual cycle-hypergraph $\Dual{\CycleHypergraph{D}}$.
We also show that we can create a generalised hypertree decomposition for $\Dual{\CycleHypergraph{D}}$ directly from a directed tree decomposition of $D$ itself and that $k$-linked sets in $D$ coincide with $k$-hyperlinked sets in $\Dual{\CycleHypergraph{D}}$.

\subsection{Directed Branchwidth}

The idea of branch decompositions for digraphs is not a new one, the notion of cyclewidth as defined in \cite{hatzel2019cyclewidth} can be seen as a, though slightly more complicated, variant of directed branchwidth.

\begin{definition}
Let $D$ be a digraph.
A \emph{directed branch decomposition} for $D$ is a tuple $\Brace{T,\delta}$ where $T$ is a subcubic tree and $\delta\colon\LineGraph{T}\rightarrow\Fkt{V}{D}$ is a bijection.
The \emph{thickness} of an edge $t_1t_2\in\Fkt{E}{T}$ is defined as follows:
\begin{align*}
	\Thickness{t_1t_2}\coloneqq\min\CondSet{\Abs{S}}{S\subseteq\Fkt{V}{D}~\text{and}~S~\text{is a hitting set for}~\Cut{\CycleHypergraph{D}}{\Fkt{\delta}{T_{t_1}}}}.
\end{align*}
The \emph{width} of a directed branch decomposition $\Brace{T,\delta}$ is defined as
\begin{align*}
\Width{\Brace{T,\delta}}\coloneqq\max_{e\in E(T)}\Thickness{e}.
\end{align*}
The \emph{\dbwText} of $D$, denoted by $\dbw{D}$, is the minimum width over all directed branch decompositions for $D$.
\end{definition}

Please note that the cut-function is symmetric and thus, in the definition above, we have $\Cut{\CycleHypergraph{D}}{T_{t_1}}=\Cut{\CycleHypergraph{D}}{T_{t_2}}$.
If $t_1t_2$ is an edge in a directed branch decomposition and $C\in\Cut{\CycleHypergraph{D}}{\Fkt{\delta}{T_{t_1}}}$, we say that $C$ \emph{crosses} $t_1t_2$.

First we will show that we can obtain a directed branch decomposition of bounded width from a directed tree decomposition.
For a digraph $D$ we call a tuple $\Brace{T,\beta,\gamma}$ a \emph{leaf directed tree decomposition}, if $\Brace{T,\beta,\gamma}$ satisfies all requirements for a being a directed tree decomposition except for the following: $\Fkt{\beta}{t}=\emptyset$ for all non-leaf vertices $t$ of $T$ and $\Abs{\Fkt{\beta}{\ell}}=1$ for all leaf vertices $\ell$ of $T$ s.t. $\Fkt{\beta}{\ell}\neq\Fkt{\beta}{\ell}$ for all distinct $\ell,\ell'\in\LineGraph{T}$, and $T$ is subcubic.
The \emph{width} of a leaf directed tree decomposition is defined analogous to the width of a directed tree decomposition.

\begin{lemma}[\cite{hatzel2019cyclewidth}]\label{lemma:leafdtw}
If $D$ is a digraph that has a directed tree decomposition of width $k$, then there exists a leaf directed tree decomposition of width $k$ for $D$.
\end{lemma}

\begin{lemma}\label{lemma:dbwleqdtw}
Let $D$ be a digraph, then $\dbw{D}\leq\dtw{D}+1$.
\end{lemma}

\begin{proof}
Let $\dtw{D}=k$, then by \cref{lemma:leafdtw}, there exists a leaf directed tree decomposition $\Brace{T,\beta,\gamma}$ of width $k$ for $D$.
Then $T$ is already a subcubic tree and $\delta$, obtained from $\beta$ by setting $\Fkt{\delta}{\ell}\coloneqq v$ if and only if $v\in\Fkt{\beta}{\ell}$, is a bijection between the leaves of $T$ and the vertices of $D$.
Moreover, let $t_1t_2\in\Fkt{E}{T}$ be any edge, then $\Fkt{\gamma}{t_1t_2}$ is a hitting set for all directed cycles in $\Cut{\CycleHypergraph{D}}{\delta{T_{t_1}}}$ by definition.
With $\Abs{\Fkt{\gamma}{t_1t_2}}\leq k+1$ our claim follows.
\end{proof}

To show that \dbwText\@ and directed treewidth are indeed equivalent, we need to bound $\dtw{D}$ in terms of $\dbw{D}$.
This is achieved by providing a winning strategy for a bounded number of cops in the directed cops and robber game.

\begin{lemma}\label{lemma:dcnleq3dbw}
If $D$ is a digraph with a directed branch decomposition of width $k$, then $\dcn{D}\leq 3k$.
\end{lemma}

\begin{proof}
We may assume $D$ to be strongly connected.
Let $\Brace{T,\delta}$ be a directed branch decomposition of width $k$ for $D$ and let $\ell\in\Fkt{V}{T}$ be an arbitrary leaf wit neighbour $t_0$ in $T$.
Then $\Thickness{\ell t_0}=1$ since $\Fkt{\delta}{\ell}$ hits all directed cycles containing $\Fkt{\delta}{\ell}$.
For every edge $e\in\Fkt{E}{T}$ let us denote by $S_e$ a minimum hitting set for the directed cycles crossing $e$, since $\Width{\Brace{T,\delta}}=k$ we know $\Abs{S_e}\leq k$ for all $e\in\Fkt{E}{T}$.

Now let us place a cop on $\Fkt{\delta}{\ell}$ as well as on every vertex in $\bigcup_{t\in N_T(t_0)\setminus\Set{\ell}}S_{t_0t}$ and denote by $X_0$ the set of all vertices occupied by cops this way.
In total we now have placed at most $3k$ cops.
For every $t\in\Fkt{N_T}{t_0}\setminus\Set{\ell}$ let $T_t$ be the subtree of $T-t_0t$ containing $t$, then no strong component of $D-X_0$ can contain vertices from $\Fkt{\delta}{T_t}$ and $\Fkt{\delta}{T_{t'}}$ simultaneously, if $t\neq t'\in\Fkt{N_T}{t_0}\setminus\Set{\ell}$.
Hence there must be a $t\in\Fkt{N_T}{t_0}\setminus\Set{\ell}$ such that the robber has chosen a strong component of $D-X_0$ contained in $\Fkt{\delta}{T_t}$.
We now derive a new cop-position $X_1\coloneqq\bigcup_{dt\in E(T)} S_{dt}$.
Since $S_{t_0t}\subseteq X_0\cap X_1$ the robber cannot leave $\Fkt{\delta}{T_t}$.
We set $t_1\coloneqq t$.

Now suppose we are in the following position:
There is an edge $t_{i-1}t_i$, the current cop position $X_i=\bigcup_{t_it'\in E(T)}S_{t_it'}$, and the robber component  $R_i$ is contained in $\Fkt{\delta}{T_{t_i}}$.
By definition of $X_i$ there cannot be distinct $t,t'\in\Fkt{N_T}{t_i}\setminus\Set{t_{i-1}}$ such that $R_i$ has vertices of both $\Fkt{\delta}{T_t}$ and $\Fkt{\delta}{T_{t'}}$, so we may assume $R_i\subseteq\Fkt{\delta}{T_t}$.
We set $t_{i+1}\coloneqq t$.
If $t_{i+1}$ is a leaf of $T$ we set $X_{i+1}\coloneqq S_{t_it_{i+1}}\cup\Set{\Fkt{\delta}{t_{i+1}}}$.
Since $S_{t_it_{i+1}}\subseteq X_i\cap X_{i+1}$ the robber cannot leave $\Fkt{\delta}{t_{i+1}}$ and thus we have captured her.
Otherwise $t_{i+1}$ is not a leaf and we set $X_{i+1}\coloneqq\bigcup_{t_{i+1}t'\in E(T)}S_{t_{i+1}t'}$.
With $X_{i+1}$ being the new cop position and $S_{t_it_{i+1}}\subseteq X_i\cap X_{i+1}$ the new robber component $R_{i+1}$ must be contained in $\Fkt{\delta}{T_{t'}}$ for some $t'\in\Fkt{N}{t_{i+1}}\setminus\Set{t_i}$.
Thus we can continue with the process.
Since $T$ is finite we will eventually catch the robber and by definition we have $\Abs{X_i}\leq 3k$ for all $i$.
\end{proof}

We can now combine \cref{lemma:dbwleqdtw,lemma:dcnleq3dbw} with \cref{thm:dhavensvsdtw} to obtain the following.

\begin{theorem}\label{thm:dbwvsdtw}
Let $D$ be a digraph, then $\dbw{D}-1\leq\dtw{D}\leq9\dbw{D}+1$.
\end{theorem}

\subsection{Of Digraphs and Hypergraphs}

Let $H$ be a hypergraph, $X\subseteq\Fkt{V}{H}$ and $F\subseteq\Fkt{E}{H}$.
Then $X$ corresponds to a set of edges of $\Dual{H}$ which we will denote by $\Dual{X}$, while $F$ corresponds to a set of vertices of $\Dual{H}$, denoted by $\Dual{F}$.

\begin{lemma}\label{lemma:hittingsets}
Let $D$ be a digraph, $X\subseteq\Fkt{V}{D}$ a non-empty set of vertices and $S\subseteq\Fkt{V}{D}$.
Then $S$ is a hitting set for $\Cut{\CycleHypergraph{D}}{X}$ if and only if $\Brace{\bigcup \Dual{X}}\cap\Brace{\bigcup\Dual{\Complement{X}}}\subseteq\bigcup\Dual{S}$.
\end{lemma}

\begin{proof}
First assume $S$ to be a hitting set for $\Cut{\CycleHypergraph{D}}{X}$ and assume there is some directed cycle $C\in\Fkt{E}{\CycleHypergraph{D}}$ with $v_C\in\Brace{\bigcup\Dual{X}}\cap\Brace{\bigcup\Dual{\Complement{X}}}$.
Then there must exist vertices $u,v\in\Fkt{V}{D}$ such that $u\in X\cap C$ and $v\in\Complement{X}\cap C$ and thus $C\in\Cut{\CycleHypergraph{D}}{X}$.
Moreover, with $S$ being a hitting set for $\Cut{\CycleHypergraph{D}}{X}$ this implies that there exists some $w\in S\cap C$ and so $v_C\in e_w\subseteq\bigcup\Dual{S}$.

For the reverse direction assume $\Brace{\bigcup \Dual{X}}\cap\Brace{\bigcup\Dual{\Complement{X}}}\subseteq\bigcup\Dual{S}$ and let $C\in\Cut{\CycleHypergraph{D}}{X}$.
Then there must exist two vertices $u,v\in C$ such that $u\in X$, $v\in\Complement{X}$ and thus $v_C\in \Brace{\bigcup \Dual{X}}\cap\Brace{\bigcup\Dual{\Complement{X}}}$.
Hence $v_C\in\bigcup\Dual{S}$ by assumption and thus there must exist a hyperedge $e_w\in \Dual{S}$ such that $v_C\in e_w$.
Hence $w\in S\cap C$ and thus $S$ must be a hitting set for $\Cut{\CycleHypergraph{D}}{X}$.
\end{proof}

\begin{theorem}\label{thm:dbwvshbw}
Let $D$ be a digraph, then $\dbw{D}=\hbw{\Dual{\CycleHypergraph{D}}}$.
\end{theorem}

\begin{proof}
Let $\Brace{T,\delta}$ be a directed branch decomposition of width $k$ for $D$.
We show that we can transform $\Brace{T,\delta}$ into  a hyperbranch decomposition for $\Dual{\CycleHypergraph{D}}$ without changing the width.
Let us define for every $\ell\in\LineGraph{T}$
\begin{align*}
	\Fkt{\Dual{\delta}}{\ell}\coloneqq e_v~\text{if and only if}~\Fkt{\delta}{\ell}=v.
\end{align*}
Now let $dt\in\Fkt{E}{T}$ be any edge and $S_{dt}$ a minimum hitting set for $\Cut{\CycleHypergraph{D}}{\Fkt{\delta}{T_d}}$.
Then we obtain from the definition of $\Dual{\delta}$ that $\Dual{\Fkt{\delta}{T_f}}=\Fkt{\Dual{\delta}}{T_f}$ for both $f\in\Set{d,t}$.
And thus, by \cref{lemma:hittingsets}, $\Brace{\bigcup \Fkt{\Dual{\delta}}{T_d}}\cap\Brace{\bigcup\Fkt{\Dual{\delta}}{T_t}}\subseteq\bigcup\Dual{S_{dt}}$.
Moreover, there cannot be an $S'\subseteq\Fkt{E}{\Dual{\CycleHypergraph{D}}}$ with the same property and $\Abs{S'}<\Abs{\Dual{S_{dt}}}$, since otherwise $\Dual{S'}$ would be a smaller hitting set for $\Cut{\CycleHypergraph{D}}{\Fkt{\delta}{T_d}}$ by \cref{lemma:hittingsets}.
Hence $\Width{\Brace{T,\Dual{\delta}}}=k$.

The proof for the reverse direction follows along similar lines and thus is omitted here.
\end{proof}

With the above theorem and the methods used to prove the equivalence between the respective branchwidth parameter and their treewidth-like counterpart we are able to obtain a directed tree decomposition for $D$ of almost optimal width from a hypertree decomposition of $\Dual{\CycleHypergraph{D}}$ of optimal width and vice versa.
The proof of the following lemma illustrates how one can obtain a generalised hypertree decomposition of bounded width for $\Dual{\CycleHypergraph{D}}$ directly from a directed tree decomposition for $D$.

While it is probably possible to obtain a directed tree decomposition for $D$ from a (generalised) hypertree decomposition for $\Dual{\CycleHypergraph{D}}$.
This step appears to be more complicated due to the requirements for the bags of a directed tree decomposition to partition the vertex set.
We believe that these complications are of purely technical nature and thus we are content with illustrating the similarities between directed treewidth and the \hwText\@ in terms of less technical approaches.

\begin{lemma}
Let $D$ be a digraph, then $\ghw{\Dual{\CycleHypergraph{D}}}\leq\dtw{D}+1$.
\end{lemma}

\begin{proof}
Let $\Brace{T,\beta,\gamma}$ be a directed tree decomposition for $D$ of width $k$.

We will start by construction a tree decomposition $\Brace{T,\Dual{\beta}}$ for $\Dual{\CycleHypergraph{D}}$ as follows:
Let $C\in\Fkt{E}{\CycleHypergraph{D}}$ be a directed cycle in $D$ and further let $B_C\subseteq\Fkt{V}{T}$ be the set of vertices $t$ of $T$ with $C\cap\Fkt{\beta}{t}\neq\emptyset$.
Then let $T_C$ be the minimum subtree of $T$ containing all vertices of $B_C$.
This provides us with a family of subtrees of $T$, one for every directed cycle in $D$.
We now define the bag function for our new tree decomposition for every $t\in\Fkt{V}{T}$
\begin{align*}
	\Fkt{\Dual{\beta}}{t}\coloneqq\Set{\left.v_c\in\Fkt{V}{\Dual{\CycleHypergraph{D}}}~\right|~t\in\Fkt{V}{T_C}}.
\end{align*}
Clearly every $T_C$ is non-empty and thus every $v_C$ occurs in at least one bag of $\Brace{T,\Dual{\beta}}$.
Moreover, for every pair $C,C'\in\Fkt{E}{\CycleHypergraph{D}}$ with $C\cap C'\neq\emptyset$ there must exist a vertex $t\in\Fkt{V}{T_C}\cap\Fkt{V}{T_{C'}}$ and thus $v_C,v_{C'}\in\Fkt{\Dual{\beta}}{t}$.
At last let $t,t'\in\Fkt{V}{T}$ such that $v_C\in\Fkt{\Dual{\beta}}{t}\cap\Fkt{\Dual{\beta}}{t'}$.
Then $t,t'\in\Fkt{V}{T_C}$ and thus the unique path with endpoints $t$ and $t'$ in $T$ must also be contained in $T_C$ implying $v_C\in\Fkt{\Dual{\beta}}{t''}$ for every vertex $t''$ on that path.
Hence $\Brace{T,\Dual{\beta}}$ is a tree decomposition for $\Dual{\CycleHypergraph{D}}$.

Next we need to define the guards.
To do this let $\Fkt{\Dual{\gamma}}{t}\coloneqq\Dual{\Fkt{\Gamma}{t}}$ for every $t\in\Fkt{V}{T}$.
Now let $v_C\in\Fkt{\Dual{\beta}}{t}$, we consider three cases:

\textbf{Case 1}: There is $\Reaches{d'}{T}{t}$, $d'\neq t$, such that $v_C\in\Fkt{\Dual{\beta}}{d'}$, then, by the definition of $\Dual{\beta}$, the directed cycle $C$ contains a vertex of $\Fkt{\beta}{T_t}$ and a vertex of $\Fkt{\beta}{T-T_t}$, where $T_t$ is the subtree of $T$ rooted at $t$.
With the existence of $d'$, $t$ is not the root of $T$ and thus there exists an edge $\Brace{d,t}\in\Fkt{E}{T}$.
Moreover, with $\Brace{T,\beta,\gamma}$ being a directed tree decomposition, $\Fkt{\gamma}{\Brace{d,t}}$ must contain a vertex of $C$.
Hence $v_C\in\bigcup\Dual{\Fkt{\Gamma}{t}}$.

\textbf{Case 2}: The first case does not hold, but there is a vertex $w'\neq t$ with $\Reaches{t}{T}{w'}$ such that $v_C\in\Fkt{\Dual{\beta}}{w'}$.
With the existence of $w'$ $t$ cannot be a leaf of $T$ and thus there exists some edge $\Brace{t,w}$ with $\Reaches{w}{T}{w'}$.
By the definition of $\Dual{\beta}$ there must be a vertex of $\Fkt{\beta}{T_w}$ and a vertex of $\Fkt{\beta}{T-T_w}$ which both are contained in $C$.
Now, the definition of directed tree decompositions implies the existence of a vertex of $C$ contained in $\Fkt{\gamma}{\Brace{t,w}}$ and thus $v_C\in\bigcup\Dual{\Fkt{\Gamma}{t}}$.

\textbf{Case 3}: Neither the first, nor the second case holds.
In this case $T_C$ only contains one vertex, which is precisely $t$.
This however implies $\Fkt{V}{C}\subseteq\Fkt{\beta}{t}\subseteq\Fkt{\Gamma}{t}$, therefore we have $v_C\in\bigcup\Dual{\Fkt{\Gamma}{t}}$.

In conclusion we obtain $\Fkt{\Dual{\beta}}{t}\subseteq\bigcup\Dual{\Fkt{\Gamma}{t}}=\bigcup\Fkt{\Dual{\gamma}}{t}$ for all $t\in\Fkt{V}{T}$ and thus $\Brace{T,\Dual{\beta},\Dual{\gamma}}$ is a generalised hypertree decomposition.
Moreover, since $\Brace{T,\beta,\gamma}$ is of width $k$, we have $\Abs{\Fkt{\Dual{\gamma}}{t}}\leq k+1$ for all $t\in\Fkt{V}{T}$ and for at least one $t$ equality holds which concludes our proof.
\end{proof}

For the next part we need to describe strong connectivity in a way that can be properly translated into the setting of the cycle hypergraph of $D$.
Let $C_0,C_1,\dots,C_{\ell}\in\CycleHypergraph{D}$, was $\Brace{C_0,C_1,\dots,C_{\ell}}$ is a \emph{chain of cycles} from $C_0$ to $C_{\ell}$ if $C_i\cap C_{i+1}\neq\emptyset$ for all $i\in\Set{0,\dots,\ell-1}$ and if $j\notin\Set{i-1,i,i+1}$, then $C_i\cap C_j=\emptyset$.

\begin{lemma}\label{lemma:strongconnect}
A digraph $D$ is strongly connected if and only if either $\Abs{\Fkt{V}{D}}=1$, or every vertex of $D$ is contained in a directed cycle and for every pair $C,C'\in\Fkt{E}{\CycleHypergraph{D}}$ there is a chain of cycles from $C$ to $C'$.
\end{lemma}

\begin{proof}
The statement is trivially true if $\Abs{\Fkt{V}{D}}=1$ and thus we may assume $\Abs{\Fkt{V}{D}}\geq 2$ in the following.

We start out with the reverse direction and assume that every vertex of $D$ is contained in a directed cycle and for every pair of directed cycles $C$ and $C'$ there is a chain of cycles from $C$ to $C'$.
Let $u,v\in\Fkt{V}{D}$ be any pair of vertices in $D$ and $C_u,C_v\in\Fkt{E}{\CycleHypergraph{D}}$ such that $x\in C_x$ for $x\in\Set{u,v}$.
Then there exists a chain of cycles from $C_u$ to $C_v$ and we can find a directed path from $u$ to $v$ and from $v$ to $u$ within that chain.
Hence $D$ must be strongly connected.

So now let us assume $D$ is strongly connected.
Let $u\in\Fkt{V}{D}$ be any vertex, then $u$ must have an out-neighbour $v$ and, since $D$ is strongly connected, there is a directed path from $v$ to $u$.
Hence this path together with the edge $\Brace{u,v}$ forms a directed cycle and thus every vertex of $D$ is contained in a directed cycle.
Moreover, the argument shows that every edge of $D$ is contained in a directed cycle.
Now let $C$ and $C'$ be directed cycles of $D$ and $P$ be a shortest directed path from $\Fkt{V}{C}$ to $\Fkt{V}{C'}$ and let $\Brace{C_0,\dots,C_{\ell}}$ be a chain of cycles with $C_0=C$ such that $C_{\ell}$ contains a vertex of $P$ closest to its endpoint in $\Fkt{V}{C'}$ among all chains of cycles in $D$.
Let $v$ be the endpoint of $P$ in $\Fkt{V}{C'}$, we claim $v\in\Fkt{V}{C_{\ell}}$.
Suppose this is not the case, then let $\Brace{x,y}$ be the edge of $P$ with $x\in\Fkt{V}{C_{\ell}}$ being the vertex closest to $v$ in $\Fkt{V}{P}\cap\Fkt{C}{C_{\ell}}$.
Let $C''$ be a directed cycle containing $\Brace{x,y}$ and $i\in\Set{0,\dots,\ell}$ the smallest integer such that $\Fkt{V}{C''}\cap\Fkt{V}{C_i}\neq\emptyset$.
Then $\Fkt{V}{C''}\cap\Fkt{V}{C_j}=\emptyset$ for all $j<i$ and thus $\Brace{C_0,\dots,C_i,C''}$ is a chain of cycles in $D$.
This however is a contradiction to our choice of $\Brace{C_0,\dots,C_{\ell}}$ and thus $v\in\Fkt{V}{C_{\ell}}$.
So now let $i\in\Set{0,\dots,\ell}$ be the smallest integer such that $\Fkt{V}{C_i}\cap\Fkt{V}{C'}\neq\emptyset$ and by the arguments above this means $\Brace{C_0=C,\dots,C_i,C'}$ is a chain of cycles from $C$ to $C'$.
\end{proof}

\begin{lemma}\label{lemma:components}
Let $D$ be a digraph and $S\subseteq\Fkt{V}{D}$.
Then $K\subseteq D-S$ is a strong component of $D-S$ if and only if $\Dual{K}\coloneqq\bigcup\Dual{\Fkt{V}{K}}\setminus\bigcup\Dual{S}$ is a connected component of $\Dual{\CycleHypergraph{D}}\setminus\bigcup\Dual{S}$.
\end{lemma}

\begin{proof}
Let us first assume $K$ is a strong component of $D-S$.
Then let $u,v\in\Fkt{V}{K}$ be two vertices.
By \cref{lemma:strongconnect} there exist directed cycles in $K$ containing $u$ and $v$.
Now let $C_u$ be a directed cycle containing $u$ and $C_v$ be a directed cycle containing $v$ in $K$.
Again by \cref{lemma:strongconnect} there exists a chain of cycles $\Brace{C_0,\dots,C_{\ell}}$ with $C_0=C_u$ and $C_{\ell}=C_v$.
Then $\Brace{v_{C_0},\dots,v_{C_{\ell}}}$ is a path in $\underline{\Dual{\CycleHypergraph{D}}\setminus\bigcup\Dual{S}}$ and thus $\Dual{K}$ is in fact a connected set in $\Dual{\CycleHypergraph{D}}\setminus\bigcup\Dual{S}$.
Suppose there is some $v_C\in\Fkt{V}{\Dual{\CycleHypergraph{D}}\setminus\bigcup\Dual{S}}\setminus\Dual{K}$ such that there is a path from $v_C$ to some $V_{C'}\in\Dual{K}$ in $\underline{\Dual{\CycleHypergraph{D}}\setminus\bigcup\Dual{S}}$.
Let $\Brace{v_{C_0},v_{C_1},\dots,v_{C_{\ell}}}$ be a shortest such path with $v_{C_0}=v_C$.
Then $v_{C_{\ell}}\in\Dual{K}$ and $v_{C_i}v_{C_j}\notin\Fkt{E}{\underline{\Dual{\CycleHypergraph{D}}\setminus\bigcup\Dual{S}}}$ for all $j\notin\Set{i-1,i,i+1}$.
Hence $C_j\cap C_i=\emptyset$ for all $i$ and $j$ as before and thus $\Brace{C_0,C_1,\dots,C_{\ell}}$ is a chain of cycles from $C$ to a directed cycle in $K$.
This however implies $C\subseteq K$ and thus $v_C\in\Dual{K}$ contradicting our assumption.
Hence $\Dual{K}$ must be a connected component of $\Dual{\CycleHypergraph{D}}\setminus\bigcup\Dual{S}$.

For the reverse direction let $\Dual{K}$ be a connected component of $\Dual{\CycleHypergraph{D}}\setminus\bigcup\Dual{S}$.
Let $X\coloneqq\Fkt{E}{\InducedSubgraph{\CycleHypergraph{D}}{\Dual{K}}}$, then $\Dual{X}\subseteq\Fkt{V}{D}$, we set $K\coloneqq\InducedSubgraph{D}{\Dual{X}}$.
Clearly every vertex of $K$ is contained in a directed cycle within $K$.
Moreover, as we have seen above, every shortest path in $\underline{\Dual{K}}$ corresponds to a chain of cycles in $K$ and thus $K$ is strongly connected.
With a similar argument one can show that $K$ is in fact a strong component of $D-S$ and thus we are done.
\end{proof}

There is a one to one correspondence between the strong components of a digraph $D$ and the connected components of $\Dual{\CycleHypergraph{D}}$.
We can now take this information to go back and forth between $k$-linked sets in digraphs and $k$-hyperlinked sets in the dual of their cycle hypergraph.
With this we conclude this section.

\begin{theorem}\label{thm:linkedvshyperlinked}
Let $D$ be a digraph.
There exists a $k$-linked set in $D$ if and only if there exists a $k$-hyperlinked set in $\Dual{\CycleHypergraph{D}}$.
\end{theorem}

\begin{proof}
Let $W$ be a $k$-linked set in $D$, we claim that $\Dual{W}$ is $k$-hyperlinked in $\Dual{\CycleHypergraph{D}}$.
Let $\Dual{S}\subseteq\Fkt{E}{\Dual{\CycleHypergraph{D}}}$ with $\Abs{S}<k$ and furthermore suppose for every connected component $\Dual{K}$ of $\Dual{\CycleHypergraph{D}}\setminus\bigcup \Dual{S}$ we have
\begin{align*}
\Abs{\CondSet{e_v\in\Dual{W}}{e_v\cap K\neq\emptyset}}\leq\frac{\Abs{\Dual{W}}}{2}=\frac{\Abs{W}}{2}.
\end{align*}
By \cref{lemma:components} there exists a strong component $K$ of $D-S$ such that $\Dual{K}=\bigcup\Dual{\Fkt{V}{K}}\setminus\bigcup\Dual{S}$.
Suppose there exists some strong component $K$ of $D-S$ such that $\Abs{\Fkt{V}{K}\cap W}>\frac{\Abs{W}}{2}$.
Then there must exist some $v\in\Fkt{V}{K}\cap W$ with $e_v\cap\Dual{K}=\emptyset$ since otherwise we would contradict the assumption above.
Since $K$ is strongly connected, \cref{lemma:strongconnect} implies the existence of a directed cycle $C$ containing $v$ and a chain of cycles from $C$ to any directed cycle $C'$ with $v_{C'}\in\Dual{K}$ which contradicts the maximality of $\Dual{K}$.
So we must have $\Abs{\Fkt{V}{K}\cap W}\leq\frac{\Abs{W}}{2}$.
This however is a contradiction to $W$ being a $k$-linked set and thus $\Dual{W}$ must be $k$-hyperlinked.
We omit the reverse direction as it can be proven completely analogous.
\end{proof}

\section{Digraphs of Directed Treewidth One}

In this section we will be concerned with precise description of strongly connected digraphs of directed treewidth $1$.
While many small classes of such digraphs are known, such as for example subdivisions of bidirected trees, to this day no complete description of directed treewidth one digraphs exists.
In order to achieve different characterisations we will make use of the observations from the previous section regarding the similarities between directed treewidth and \hwText.
A major advantage this approach brings with it is the fact that the class of \hwText\@ one hypergraphs is well described \cite{gottlob2002hypertree}.

Structural matching theory has developed a rich set of tools including certain decomposition techniques and minor operations that have found applications in digraph theory as well.
The tools developed to better understand so called \emph{matching covered graphs}, i.\@ e.\@ connected graphs where every edge is contained in a perfect matching, have been of great use especially in the theory of butterfly minors \cite{guenin2011packing,millani2019colouring}.
Many of the techniques we will use in order to obtain a forbidden minor characterisation of strongly connected digraphs of directed treewidth one are direct translations or at least inspired by matching theory.

A graph is called \emph{chordal} if it does not contain an induced cycle of length $\geq 4$.
A hypergraph is \emph{conformal} if a set $X\subseteq\Fkt{V}{H}$ induces a maximal clique in $\underline{H}$ if and only if $X\in\Fkt{V}{E}$.
If $H$ is conformal and $\underline{H}$ is chordal we call $H$ \emph{acyclic}.
This notion of acyclicity, sometimes known as \emph{$\alpha$-acyclicity}, has its roots in database theory \cite{fagin1983degrees,beeri1983desirability} and was one of the main concepts behind the definition of \hwText  \cite{gottlob2002hypertree}.
A hypergraph $H$ is a \emph{hypertree} if there exists a tree $T$ with $\Fkt{V}{T}=\Fkt{V}{H}$ such that every edge of $H$ induces a subtree of $T$.
Moreover, $H$ is said to have the \emph{Helly-property}, if every family $\mathcal{F}\subseteq\Fkt{E}{H}$ of pairwise intersecting edges satisfies $\bigcap\mathcal{F}\neq\emptyset$.

\begin{theorem}[\cite{fagin1983degrees,berge1984hypergraphs,gottlob2002hypertree}]\label{thm:hypertrees}
Let $H$ be a hypergraph.
The following statements are equivalent:
\begin{enumerate}
	
	\item $H$ is a hypertree,
	
	\item $H$ has the Helly-property and $\LineGraph{H}$ is chordal,
	
	\item $\Dual{H}$ is conformal and $\underline{H}$ is chordal,
	
	\item $\Dual{H}$ is $\alpha$-acyclic, and
	
	\item $\hw{\Dual{H}}=1$.
	
\end{enumerate}
\end{theorem}

We are ready to state the main result of this section.

\begin{theorem}\label{thm:dtw1}
Let $D$ be a non-trivial strongly connected digraph, the following statements are equivalent:
\begin{enumerate}
	
	\item $\dtw{D}=1$,
	
	\item $\CycleHypergraph{D}$ is a hypertree, and
	
	\item $D$ does neither contain a bicycle, nor $A_4$, see \cref{fig:A4}, as a butterfly minor.
	
\end{enumerate}
\end{theorem}

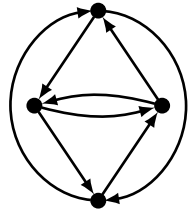
\begin{figure}[h!]
	\begin{center}
		\begin{tikzpicture}[scale=0.7]
		
		\pgfdeclarelayer{background}
		\pgfdeclarelayer{foreground}
		
		\pgfsetlayers{background,main,foreground}
		
		\tikzstyle{v:main} = [draw, circle, scale=0.5, thick,fill=black]
		\tikzstyle{v:tree} = [draw, circle, scale=0.3, thick,fill=black]
		\tikzstyle{v:border} = [draw, circle, scale=0.75, thick,minimum size=10.5mm]
		\tikzstyle{v:mainfull} = [draw, circle, scale=1, thick,fill]
		\tikzstyle{v:ghost} = [inner sep=0pt,scale=1]
		\tikzstyle{v:marked} = [circle, scale=1.2, fill=CornflowerBlue,opacity=0.3]
		
		\tikzset{>=latex} 
		\tikzstyle{e:marker} = [line width=9pt,line cap=round,opacity=0.2,color=DarkGoldenrod]
		\tikzstyle{e:colored} = [line width=1.2pt,color=BostonUniversityRed,cap=round,opacity=0.8]
		\tikzstyle{e:coloredthin} = [line width=1.1pt,opacity=0.8]
		\tikzstyle{e:coloredborder} = [line width=2pt]
		\tikzstyle{e:main} = [line width=1pt]
		\tikzstyle{e:extra} = [line width=1.3pt,color=LavenderGray]
		
		\begin{pgfonlayer}{main}
		
		\node (C) [] {};
		
		\node (C1) [v:ghost, position=180:25mm from C] {};
		
		\node (C2) [v:ghost, position=0:0mm from C] {};
		
		\node (C3) [v:ghost, position=0:25mm from C] {};

		

		
		
		\node (v1) [v:main,position=90:18mm from C2] {};
		\node (v2) [v:main,position=180:12mm from C2] {};
		\node (v3) [v:main,position=270:18mm from C2] {};
		\node (v4) [v:main,position=0:12mm from C2] {};
		
		\node (g1) [v:ghost,position=180:4.5mm from v2] {};
		\node (g2) [v:ghost,position=0:4.5mm from v4] {};
		
		\node (g3) [v:ghost,position=180:1mm from v1] {};
		\node (g4) [v:ghost,position=0:1mm from v3] {};
		
		
		
		

		

		
		
		\draw (v2) [e:main,->,bend right=15] to (v4);
		\draw (v4) [e:main,->,bend right=15] to (v2);
		
		\draw (v1) edge[line width=1pt,quick curve through={(g2)},->] (g4);
		\draw (v3) edge[line width=1pt,quick curve through={(g1)},->] (g3);
		
		\draw (v1) [e:main,->] to (v2);
		\draw (v2) [e:main,->] to (v3);
		\draw (v3) [e:main,->] to (v4);
		\draw (v4) [e:main,->] to (v1);
		
		

		
		
		\end{pgfonlayer}
		

		\begin{pgfonlayer}{background}
		
		\end{pgfonlayer}	
		
		\begin{pgfonlayer}{foreground}

		\end{pgfonlayer}
		\end{tikzpicture}
	\end{center}
	\caption{The digraph $A_4$.}
	\label{fig:A4}
\end{figure}

In particular we obtain $\dtw{D}=1$ if and only if $\hw{\Dual{\CycleHypergraph{D}}}=1$ from \cref{thm:hypertrees}.

Crucial for the proofs in this section is a digraphic analogue of the so called \emph{tight cut decomposition} \cite{lovasz1987matching} which sits at the heart of the vast majority of matching theoretic results.

Let $D$ be a digraph.
A \emph{directed separation} in $D$ is a tuple $\Separation{A}{B}{}{}$ where $A\cup B=\Fkt{V}{D}$ and there is no edge with tail in $B\setminus A$ and head in $A\setminus B$, we call $A$ and $B$ the \emph{shores} of $\Separation{A}{B}{}{}$.
The set $A\cap B$ is called the \emph{separator} of $\Separation{A}{B}{}{}$ and its \emph{order}, denoted by $\Abs{\Separation{A}{B}{}{}}$, is defined as $\Abs{A\cap B}$.
Two directed separations $\Separation{A}{B}{}{}$ and $\Separation{C}{D}{}{}$ are said to \emph{cross} if the following sets all are non-empty:
\begin{align*}
A\cap C,~B\cap D,~\Brace{A\cap D}\setminus\Brace{B\cap C},~\text{and}~\Brace{B\cap C}\setminus\Brace{A\cap D}.
\end{align*}
If $\Separation{A}{B}{}{}$ and $\Separation{C}{D}{}{}$ do not cross they are called \emph{laminar}.
A directed separation of order $1$ will be called a \emph{tight separation} in the following as a tribute to the tight cut decomposition, it is \emph{non-trivial} if $\Abs{A}\geq 2$ and $\Abs{B}\geq 2$.

Let $\Separation{A}{B}{}{}$ be a tight separation in $D$, $\Set{v}=A\cap B$, and let $X\in\Set{A,B}$.
We call $v$ the \emph{cut vertex} of $\Separation{A}{B}{}{}$.
The \emph{tight separation contraction} of $X$ in $D$ is defined as the digraph $D'$ obtained from $D$ by identifying the whole set of $X$ into the vertex $v$ and identifying multiple edges.

Observe that if $\Separation{A}{B}{}{}$ is a non-trivial tight separation in $D$, $D'$ is obtained from $D$ by contracting some $X\in\Set{A,B}$ and $\Separation{Y'}{Z'}{}{}$ is a tight separation in $D'$, then $\Separation{Y'}{Z'}{}{}$ naturally corresponds to a non-trivial tight separation $\Separation{Y}{Z}{}{}$ in $D$ that is laminar with $\Separation{A}{B}{}{}$.
Moreover, a digraph $D$ is \emph{strongly $k$-connected} if $D-S$ is strongly connected for all $S\subseteq\Fkt{V}{D}$ of size at most $k-1$.
So $D$ is strongly $2$-connected if and only if it does not have a non-trivial tight separation.
A \emph{dibrace} of $D$ is a strongly $2$-connected digraph that can be obtained from $D$ by only applying tight separation contractions.

Please note that any tight separation contraction $D'$ of a digraph $D$ is in fact a butterfly minor of $D$.
To see this suppose $D'$ can be obtained from $D$ be contracting the side $X\in\Set{A,B}$ in the non trivial tight separation contraction $\Separation{A}{B}{}{}$.
Without loss of generality let us assume $X=B$ and let $v\in A\cap B$ be the cut vertex.
Now let $T$ be the result of a depth first search in $\InducedSubgraph{D}{B}$ starting at $v$.
At last delete all edges in $\Fkt{E}{\InducedSubgraph{D}{B}}\setminus\Fkt{E}{T}$ and contract all of $T$ into one vertex.
The result is exactly $D'$.

\begin{definition}
	Let $D$ be a digraph and $\mathcal{S}$ a maximal family of pairwise laminar non-trivial tight separations in $D$.
	An $\mathcal{S}$-decomposition for $D$ is a tuple $\Brace{T,\sigma,\zeta,\mathcal{S}}$ where $T$ is a tree, $\sigma\colon\Fkt{E}{T}\rightarrow\mathcal{S}$ a bijection, and $\zeta\colon \Fkt{V}{T}\times\Fkt{E}{T}\rightarrow 2^{V(D)}$ is a partial function such that
	\begin{enumerate}
		
		\item for every $dt\in\Fkt{E}{T}$ we have $\Fkt{\sigma}{dt}\in\Set{\Separation{\Fkt{\zeta}{d,dt}}{\Fkt{\zeta}{t,dt}}{}{},\Separation{\Fkt{\zeta}{t,dt}}{\Fkt{\zeta}{d,dt}}{}{}}$, and
		
		\item for every vertex $t\in\Fkt{V}{T}$ the digraph $D_t$ obtained from $D$ by contracting every $\Fkt{\zeta}{d,dt}$ where $dt\in\Fkt{E}{T}$ is a dibrace of $D$.
		
	\end{enumerate}
For every $t\in\Fkt{V}{T}$ we define $B_t\coloneqq \bigcap_{dt\in E(T)}\Fkt{\zeta}{t,dt}$.
\end{definition}

We say that a vertex $v\in\Fkt{V}{D}$ is \emph{butterfly dominating} if every butterfly contractible edge in $D$ is incident with $v$ and if there are both incoming and outgoing butterfly contractible edges at $v$, then $v$ either as exactly one incoming, or one outgoing edge.
Note that if $D$ does not have any butterfly contractible edge every vertex is butterfly dominating.
The following theorem is a digraphic analogue of Theorem 29 from \cite{mccuaig2001brace}.

\begin{theorem}\label{thm:strongminors}
Every strongly $2$-connected digraph on at least $3$ vertices contains a bicycle or $A_4$ as a butterfly minor.
\end{theorem}

\begin{proof}
We will, in fact, prove a slightly stronger statement by induction.

\textbf{Claim}: Let $D$ be a strongly connected digraph on at least $3$ vertices such that there exists a butterfly dominating vertex $x\in\Fkt{V}{D}$, then $D$ has a bicycle or $A_4$ as a butterfly minor.

Suppose $\Abs{\Fkt{V}{D}}=3$, then $D$ must be isomorphic to one of the digraphs from \cref{fig:strongly3vertices}.

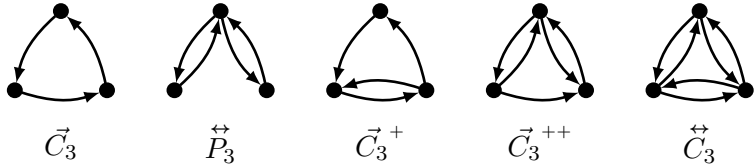
\begin{figure}[h!]
	\begin{center}
		\begin{tikzpicture}[scale=0.7]
		
		\pgfdeclarelayer{background}
		\pgfdeclarelayer{foreground}
		
		\pgfsetlayers{background,main,foreground}
		
		\tikzstyle{v:main} = [draw, circle, scale=0.5, thick,fill=black]
		\tikzstyle{v:tree} = [draw, circle, scale=0.3, thick,fill=black]
		\tikzstyle{v:border} = [draw, circle, scale=0.75, thick,minimum size=10.5mm]
		\tikzstyle{v:mainfull} = [draw, circle, scale=1, thick,fill]
		\tikzstyle{v:ghost} = [inner sep=0pt,scale=1]
		\tikzstyle{v:marked} = [circle, scale=1.2, fill=CornflowerBlue,opacity=0.3]
		
		\tikzset{>=latex} 
		\tikzstyle{e:marker} = [line width=9pt,line cap=round,opacity=0.2,color=DarkGoldenrod]
		\tikzstyle{e:colored} = [line width=1.2pt,color=BostonUniversityRed,cap=round,opacity=0.8]
		\tikzstyle{e:coloredthin} = [line width=1.1pt,opacity=0.8]
		\tikzstyle{e:coloredborder} = [line width=2pt]
		\tikzstyle{e:main} = [line width=1pt]
		\tikzstyle{e:extra} = [line width=1.3pt,color=LavenderGray]
		
		\begin{pgfonlayer}{main}
		
		\node (C) [] {};
		
		\node (C1) [v:ghost, position=180:60mm from C] {};
		\node (L1) [v:ghost, position=270:15mm from C1,align=center] {$\vec{C_3}$};
		
		\node (C2) [v:ghost, position=180:30mm from C] {};
				\node (L2) [v:ghost, position=270:15mm from C2,align=center] {$\Bidirected{P_3}$};
		
		\node (C3) [v:ghost, position=0:0mm from C] {};
				\node (L3) [v:ghost, position=270:15mm from C3,align=center] {$\vec{C_3}^{\text{\tiny +}}$};
		
		\node (C4) [v:ghost, position=0:30mm from C] {};
				\node (L4) [v:ghost, position=270:15mm from C4,align=center] {$\vec{C_3}^{\text{\tiny ++}}$};
		
		\node (C5) [v:ghost, position=0:60mm from C] {};
		\node (L5) [v:ghost, position=270:15mm from C5,align=center] {$\Bidirected{C_3}$};

		

		
		
		\node (v1-1) [v:main,position=90:10mm from C1] {};
		\node (v2-1) [v:main,position=210:10mm from C1] {};
		\node (v3-1) [v:main,position=330:10mm from C1] {};

		\node (v1-2) [v:main,position=90:10mm from C2] {};
		\node (v2-2) [v:main,position=210:10mm from C2] {};
		\node (v3-2) [v:main,position=330:10mm from C2] {};
		
		\node (v1-3) [v:main,position=90:10mm from C3] {};
		\node (v2-3) [v:main,position=210:10mm from C3] {};
		\node (v3-3) [v:main,position=330:10mm from C3] {};
		
		\node (v1-4) [v:main,position=90:10mm from C4] {};
		\node (v2-4) [v:main,position=210:10mm from C4] {};
		\node (v3-4) [v:main,position=330:10mm from C4] {};
		
		\node (v1-5) [v:main,position=90:10mm from C5] {};
		\node (v2-5) [v:main,position=210:10mm from C5] {};
		\node (v3-5) [v:main,position=330:10mm from C5] {};

		
		
		

		

		
		
		\draw (v1-1) [e:main,->,bend right=18] to (v2-1);
		\draw (v2-1) [e:main,->,bend right=18] to (v3-1);
		\draw (v3-1) [e:main,->,bend right=18] to (v1-1);

		\draw (v1-2) [e:main,->,bend right=18] to (v2-2);
		\draw (v2-2) [e:main,->,bend right=18] to (v1-2);
		\draw (v1-2) [e:main,->,bend right=18] to (v3-2);
		\draw (v3-2) [e:main,->,bend right=18] to (v1-2);
		
		\draw (v1-3) [e:main,->,bend right=18] to (v2-3);
		\draw (v2-3) [e:main,->,bend right=18] to (v3-3);
		\draw (v3-3) [e:main,->,bend right=18] to (v1-3);
		\draw (v3-3) [e:main,->,bend right=18] to (v2-3);
		
		\draw (v1-4) [e:main,->,bend right=18] to (v2-4);
		\draw (v2-4) [e:main,->,bend right=18] to (v3-4);
		\draw (v3-4) [e:main,->,bend right=18] to (v1-4);
		\draw (v2-4) [e:main,->,bend right=18] to (v1-4);
		\draw (v1-4) [e:main,->,bend right=18] to (v3-4);
		
		\draw (v1-5) [e:main,->,bend right=18] to (v2-5);
		\draw (v2-5) [e:main,->,bend right=18] to (v3-5);
		\draw (v3-5) [e:main,->,bend right=18] to (v1-5);
		\draw (v2-5) [e:main,->,bend right=18] to (v1-5);
		\draw (v1-5) [e:main,->,bend right=18] to (v3-5);
		\draw (v3-5) [e:main,->,bend right=18] to (v2-5);

		

		
		
		\end{pgfonlayer}
		

		\begin{pgfonlayer}{background}
		
		\end{pgfonlayer}	
		
		\begin{pgfonlayer}{foreground}

		\end{pgfonlayer}
		\end{tikzpicture}
	\end{center}
	\caption{The strongly connected digraphs on $3$ vertices.}
	\label{fig:strongly3vertices}
\end{figure}

None of these graphs except for $\Bidirected{C_3}$ has a butterfly dominating vertex and thus we are done with the base of the induction.
So from now on we may assume $\Abs{\Fkt{V}{D}}\geq 4$ and let $x\in\Fkt{V}{D}$ be a butterfly dominating vertex of $D$.
We will consider several cases.

\textbf{Case 1}: $D$ is not strongly $2$-connected.

Under this assumption there must exist a vertex $v\in\Fkt{V}{D}$ such that $D-v$ is not strongly connected.
Let $K$ be the strong component of $D-v$ containing $x$ if $v\neq x$, otherwise let $K$ be any strong component of $D-v$.
If there exists a strong component $K'\neq K$ of $D-v$ such that there is a directed path from $K'$ to $K$ in $D-v$ let $X$ be the union of all strong components, including $K$ itself, that can be reached by a directed path starting in $K$ and let $Y\coloneqq\Fkt{V}{D-v}\setminus X$.
In case no such component exists let $X\coloneqq\Fkt{V}{K}$ and $Y$ be defined as above.
In either case one of $\Separation{Y\cup\Set{v}}{X\cup\Set{v}}{}{}$ and $\Separation{X\cup\Set{v}}{Y\cup\Set{v}}{}{}$ is a non-trivial tight separation.
Thus we may contract the set $X$ into the vertex $v$, let $D'$ be the result of this operation.
We claim that $v$ is butterfly dominating in $D'$.
Suppose there is a butterfly contractible edge in $D'$ not incident with $v$, then this edge must have been butterfly contractible in $D$ as well, contradicting the choice of $v$.
Furthermore, suppose $v$ has an incoming and an outgoing butterfly contractible edge, but $\Abs{\InN{D'}{v}}\geq 2$ and $\Abs{\OutN{D'}{v}}\geq 2$.
Then $v$ is neither the only tail, nor the only head of any of its incident butterfly contractible edges in $D'$.
This however implies the existence of some butterfly contractible edge $e'$ that corresponds to an edge $e$ in $D$ that must still be butterfly contractible, but is not incident to $v$.
Thus $v$ cannot be butterfly dominating in $D$.
Hence $v$ must be butterfly dominating in $D'$ and thus we may apply the induction hypothesis to $D'$.
With $D'$ being a butterfly minor of $D$ the assertion follows.

So from now on we may assume $D$ to be strongly $2$ connected.

\textbf{Case 2}: There exists a vertex $u\in\Fkt{V}{D}$ such that $u$ has at least three out-neighbours or at least three in-neighbours.

Without loss of generality let us assume that $u$ has at least three out-neighbours and let $\Brace{u,w}\in\Fkt{E}{D}$ be any edge incident with $u$.
Since $D$ is strongly $2$-connected, $D'\coloneqq D-\Brace{u,w}$ is strongly connected.
Indeed $\Abs{\OutN{D'}{x}}\geq 2$ for all $x\in\Fkt{V}{D'}$ and $\Abs{\InN{D'}{x}}\geq 2$ for all $x\in\Fkt{V}{D'}\setminus\Set{w}$, hence $w$ is butterfly dominating in $D'$ and we may apply the induction hypothesis.
With $D'\subseteq D$ this immediately yields the assertion.

Thus we may further assume $\Abs{\InN{D}{x}}=\Abs{\OutN{D}{x}}=2$ for all $x\in\Fkt{V}{D}$.
We say that a digraph $D$ has the \emph{small cycle property} if for every edge $e\in\Fkt{E}{D}$ there exists an induced subgraph of $D$ isomorphic to $\Bidirected{K_2}$, $\vec{K_3}$, $\vec{K_3}^{\text{\tiny +}}$, $\vec{K_3}^{\operatorname{o}}$, $\vec{K_3}^{\operatorname{i}}$, $\vec{K_3}^{\text{\tiny ++}}$,or $K_{2,2}^{\uparrow}$ containing $e$.
See \cref{fig:smallcycles} for an illustration.

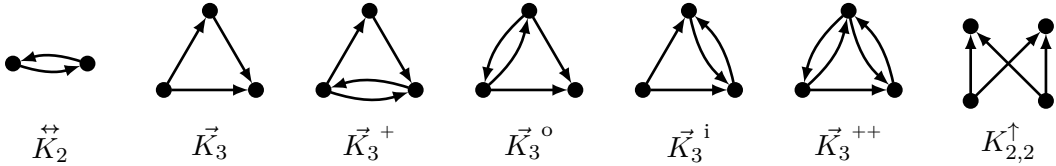
\begin{figure}[h!]
	\begin{center}
		\begin{tikzpicture}[scale=0.7]
		
		\pgfdeclarelayer{background}
		\pgfdeclarelayer{foreground}
		
		\pgfsetlayers{background,main,foreground}
		
		\tikzstyle{v:main} = [draw, circle, scale=0.5, thick,fill=black]
		\tikzstyle{v:tree} = [draw, circle, scale=0.3, thick,fill=black]
		\tikzstyle{v:border} = [draw, circle, scale=0.75, thick,minimum size=10.5mm]
		\tikzstyle{v:mainfull} = [draw, circle, scale=1, thick,fill]
		\tikzstyle{v:ghost} = [inner sep=0pt,scale=1]
		\tikzstyle{v:marked} = [circle, scale=1.2, fill=CornflowerBlue,opacity=0.3]
		
		\tikzset{>=latex} 
		\tikzstyle{e:marker} = [line width=9pt,line cap=round,opacity=0.2,color=DarkGoldenrod]
		\tikzstyle{e:colored} = [line width=1.2pt,color=BostonUniversityRed,cap=round,opacity=0.8]
		\tikzstyle{e:coloredthin} = [line width=1.1pt,opacity=0.8]
		\tikzstyle{e:coloredborder} = [line width=2pt]
		\tikzstyle{e:main} = [line width=1pt]
		\tikzstyle{e:extra} = [line width=1.3pt,color=LavenderGray]
		
		\begin{pgfonlayer}{main}
		
		\node (C) [] {};
		
		\node (C1) [v:ghost, position=180:90mm from C] {};
		\node (L1) [v:ghost, position=270:15mm from C1,align=center] {$\Bidirected{K_2}$};
		
		\node (C2) [v:ghost, position=180:60mm from C] {};
		\node (L2) [v:ghost, position=270:15mm from C2,align=center] {$\vec{K_3}$};
		
		\node (C3) [v:ghost, position=180:30mm from C] {};
		\node (L3) [v:ghost, position=270:15mm from C3,align=center] {$\vec{K_3}^{\text{\tiny +}}$};
		
		\node (C4) [v:ghost, position=0:0mm from C] {};
		\node (L4) [v:ghost, position=270:15mm from C4,align=center] {$\vec{K_3}^{\operatorname{o}}$};
		
		\node (C5) [v:ghost, position=0:30mm from C] {};
		\node (L5) [v:ghost, position=270:15mm from C5,align=center] {$\vec{K_3}^{\operatorname{i}}$};
		
		\node (C6) [v:ghost, position=0:60mm from C] {};
		\node (L6) [v:ghost, position=270:15mm from C6,align=center] {$\vec{K_3}^{\text{\tiny ++}}$};
		
		\node (C7) [v:ghost, position=0:90mm from C] {};
		\node (L7) [v:ghost, position=270:15mm from C7,align=center] {$K_{2,2}^{\uparrow}$};

		

		
		
		\node (v1-1) [v:main,position=180:7mm from C1] {};
		\node (v2-1) [v:main,position=0:7mm from C1] {};
		
		\node (v1-2) [v:main,position=90:10mm from C2] {};
		\node (v2-2) [v:main,position=210:10mm from C2] {};
		\node (v3-2) [v:main,position=330:10mm from C2] {};
		
		\node (v1-3) [v:main,position=90:10mm from C3] {};
		\node (v2-3) [v:main,position=210:10mm from C3] {};
		\node (v3-3) [v:main,position=330:10mm from C3] {};
		
		\node (v1-4) [v:main,position=90:10mm from C4] {};
		\node (v2-4) [v:main,position=210:10mm from C4] {};
		\node (v3-4) [v:main,position=330:10mm from C4] {};
		
		\node (v1-5) [v:main,position=90:10mm from C5] {};
		\node (v2-5) [v:main,position=210:10mm from C5] {};
		\node (v3-5) [v:main,position=330:10mm from C5] {};
		
		\node (v1-6) [v:main,position=90:10mm from C6] {};
		\node (v2-6) [v:main,position=210:10mm from C6] {};
		\node (v3-6) [v:main,position=330:10mm from C6] {};
		
		\node (v1-7) [v:main,position=135:10mm from C7] {};
		\node (v2-7) [v:main,position=225:10mm from C7] {};
		\node (v3-7) [v:main,position=315:10mm from C7] {};
		\node (v4-7) [v:main,position=45:10mm from C7] {};

		
		
		

		

		
		
		\draw (v1-1) [e:main,->,bend right=18] to (v2-1);
		\draw (v2-1) [e:main,->,bend right=18] to (v1-1);
		
		\draw (v2-2) [e:main,->] to (v1-2);
		\draw (v1-2) [e:main,->] to (v3-2);
		\draw (v2-2) [e:main,->] to (v3-2);
		
		\draw (v2-3) [e:main,->] to (v1-3);
		\draw (v1-3) [e:main,->] to (v3-3);
		\draw (v3-3) [e:main,->,bend right=18] to (v2-3);
		\draw (v2-3) [e:main,->,bend right=18] to (v3-3);
		
		\draw (v1-4) [e:main,->,bend right=18] to (v2-4);
		\draw (v2-4) [e:main,->,bend right=18] to (v1-4);
		\draw (v1-4) [e:main,->] to (v3-4);
		\draw (v2-4) [e:main,->] to (v3-4);
		
		\draw (v2-5) [e:main,->] to (v1-5);
		\draw (v2-5) [e:main,->] to (v3-5);
		\draw (v1-5) [e:main,->,bend right=18] to (v3-5);
		\draw (v3-5) [e:main,->,bend right=18] to (v1-5);
		
		\draw (v2-6) [e:main,->,bend right=18] to (v1-6);
		\draw (v1-6) [e:main,->,bend right=18] to (v2-6);
		\draw (v2-6) [e:main,->] to (v3-6);
		\draw (v1-6) [e:main,->,bend right=18] to (v3-6);
		\draw (v3-6) [e:main,->,bend right=18] to (v1-6);
		
		\draw (v2-7) [e:main,->] to (v1-7);
		\draw (v2-7) [e:main,->] to (v4-7);
		\draw (v3-7) [e:main,->] to (v1-7);
		\draw (v3-7) [e:main,->] to (v4-7);

		

		
		
		\end{pgfonlayer}
		

		\begin{pgfonlayer}{background}
		
		\end{pgfonlayer}	
		
		\begin{pgfonlayer}{foreground}

		\end{pgfonlayer}
		\end{tikzpicture}
	\end{center}
	\caption{The induced subgraphs for the small cycle property.}
	\label{fig:smallcycles}
\end{figure}

\textbf{Case 3}: $D$ does not have the small cycle property.

Then there must exist an edge $\Brace{x,y}\in\Fkt{E}{D}$ for which no induced subgraph of $D$ containing it belongs to one of the six from \cref{fig:smallcycles}.
By assumption there exists a unique out-neighbour $z\in\OutN{D}{x}$ of $x$ that is distinct from $y$.
In $D-\Brace{x,z}$ an edge is butterfly contractible if and only if it is an incoming edge of $z$ or the edge $\Brace{x,y}$.
Let now be $D'$ the digraph obtained from $D-\Brace{x,z}$ by contracting $\Brace{x,y}$.
In $D'$ every vertex except $z$ has out- and in-degree at least two and thus $z$ is butterfly dominating.
Since $D'$ must also be strongly connected the assertion follows immediately from our induction hypothesis.

So $D$ must, in addition to its other properties, also have the small cycle property.

\textbf{Case 4}: $D$ contains an induced subgraph isomorphic to $\vec{K_3}$.

Let $x,y,z$ be the vertices and $\Brace{x,y}$, $\Brace{y,z}$, $\Brace{x,z}$ be the edges of a $\vec{K_3}$.
Then $D-\Brace{x,z}$ is strongly connected and contains precisely two butterfly contractible edges, namely $\Brace{x,y}$ and $\Brace{y,z}$.
Now let $D'$ be the digraph obtained from $D-\Brace{x,z}$ by contracting the edge $\Brace{y,z}$ into the vertex $y$.
So in $D'$ every vertex has in-degree at least two and every vertex except $x$ also has out-degree at least two.
Hence $x$ is butterfly dominating in $D'$ and we may apply the induction hypothesis to close this case.

\textbf{Case 5}: $D$ contains an induced subgraph isomorphic to $\vec{K_3}^{\text{\tiny +}}$.

Let $x,y,z$ be the vertices and $\Brace{x,y}$, $\Brace{y,z}$, $\Brace{x,z}$, $\Brace{z,x}$ be the edges of a $\vec{K_3}^{\text{\tiny +}}$.

First suppose $\Abs{\Fkt{V}{D}}=4$, then , since $\vec{K_3}^{\text{\tiny +}}$ is induced, there must exist a fourth vertex $w$ together with the edges $\Brace{w,x}$ and $\Brace{z,w}$.
Moreover, by our degree constraints $w$ and $y$ must be contained in a digon together.
Then $D$ is isomorphic to $A_4$ and we are done.

So we may assume $\Abs{\Fkt{V}{D}}\geq 5$.
Let us consider the strongly connected digraph $D-\Brace{x,z}-\Brace{z,x}$.
Here the butterfly contractible edges are exactly the four edges incident with $x$ or $z$.
So in particular the edges $\Brace{x,y}$ and $\Brace{y,z}$ are butterfly contractible.
Now let $D'$ be the digraph obtained from $D-\Brace{x,z}-\Brace{z,x}$ by contracting both $\Brace{x,y}$ and $\Brace{y,z}$ into the vertex $y$.
Now $y$ has out- and in-degree exactly $2$ in $D'$ and so do the former neighbours of $x$ and $z$, thus there is no more butterfly contractible edge in the graph and we can close this case by applying our induction hypothesis.

\textbf{Case 6}: $D$ contains an induced subgraph isomorphic to $\vec{K_3}^{\operatorname{o}}$ or $\vec{K_3}^{\operatorname{i}}$.

We only consider one of these cases since the other one follows analogously.
So let us assume $\vec{K_3}^{\operatorname{o}}$ is an induced subgraph of $D$ and $z$ is the vertex of $\vec{K_3}^{\operatorname{o}}$ not contained in a digon.
Let $x,y$ be the two other vertices.
Observe that every path staring in $u\in\Set{x,y}$ and ending in a vertex of $\Fkt{V}{D}\setminus\Set{x,y,z}$ must contain $z$.
Hence $D-z$ is not strongly connected which is a contradiction.
 
 \textbf{Case 7}: $D$ contains an induced subgraph isomorphic to $\vec{K_3}^{\text{\tiny ++}}$.
 
 This case is similar to the previous one.
 Let $\Brace{x,y}$ be the unique edge of $\vec{K_3}^{\text{\tiny ++}}$ not contained in a digon, then every path from $x$ to any vertex of $\Fkt{V}{D}-\Fkt{V}{\vec{K_3}^{\text{\tiny ++}}}$ must contain $y$ and thus $D-y$ is not strongly connected.
 
\textbf{Case 8}: $D$ contains an induced subgraph isomorphic to $K_{2,2}^{\uparrow}$.

Let $w,x,y,z\in\Fkt{V}{D}$ together with the edges $\Brace{w,y}$, $\Brace{w,z}$, $\Brace{x,y}$, $\Brace{x,z}$ form an induced $K_{2,2}^{\uparrow}$ in $D$.
Consider $D-\Brace{w,y}$ which contains precisely two butterfly contractible edges, namely $\Brace{w,z}$ and $\Brace{x,y}$.
Similar to the previous cases let now be $D'$ the digraph obtained from $D-\Brace{w,y}$ by contracting $\Brace{w,z}$ and observe that $\Brace{x,y}$ is the sole butterfly contractible edge of $D'$.
Hence we may again apply out induction hypothesis and are done with the current case.

To summarise, $D$ is strongly $2$-connected, has maximum in- and out-degree equal to two, has the small cycle property and cannot contain an induced `small cycle' except for $\Bidirected{K_2}$.
This means that every every edge of $D$ is contained in a digon, hence $D$ is the biorientation of some undirected graph $G$.
Moreover, since every vertex of $D$ has out- and in-degree exactly two, $G$ must have $\Fkt{\delta}{G}=\Fkt{\Delta}{G}=2$ and with $D$ being strongly $2$-connected, $D$ must be $2$-connected. 
Hence $G$ must be isomorphic to some $C_{\ell}$, $\ell\geq 3$ and thus $D$ is a bicycle.
\end{proof}

The remainder of the section is dedicated to the proof of \cref{thm:dtw1}.
The proof is divided into several smaller lemmas.

Let $C_1,\dots, C_{\ell}$ be a collection of pairwise distinct directed cycles in $D$.
We call $\Brace{C_1,\dots,C_{\ell}}$ a \emph{closed chain of cycles}, if $\Brace{C_1,\dots,C_{\ell-1}}$ and $\Brace{C_2,\dots,C_{\ell}}$ are chains of cycles and $\Fkt{V}{C_1}\cap\Fkt{V}{C_{\ell}}\neq\emptyset$.

\begin{lemma}\label{lemma:smallhavens}
If a digraph $D$ contains a bicycle or $A_4$ as a butterfly minor, it contains a haven of order $3$.
\end{lemma}

\begin{proof}
Let us assume $D$ contains a bicycle of length $\ell$ as a butterfly minor, then there exists a closed chain of cycles $\Brace{C_1,\dots,C_{\ell}}$ in $D$.
Similarly if $D$ contains $A_4$ as a butterfly minor we can find a closed chain of $3$ cycles in $D$ as illustrated in \cref{fig:A4chain}.

\begin{figure}[h!]
	\begin{center}
		\begin{tikzpicture}[scale=0.7]
		
		\pgfdeclarelayer{background}
		\pgfdeclarelayer{foreground}
		
		\pgfsetlayers{background,main,foreground}
		
		\tikzstyle{v:main} = [draw, circle, scale=0.5, thick,fill=black]
		\tikzstyle{v:tree} = [draw, circle, scale=0.3, thick,fill=black]
		\tikzstyle{v:border} = [draw, circle, scale=0.75, thick,minimum size=10.5mm]
		\tikzstyle{v:mainfull} = [draw, circle, scale=1, thick,fill]
		\tikzstyle{v:ghost} = [inner sep=0pt,scale=1]
		\tikzstyle{v:marked} = [circle, scale=1.2, fill=CornflowerBlue,opacity=0.3]
		
		\tikzset{>=latex} 
		\tikzstyle{e:marker} = [line width=9pt,line cap=round,opacity=0.2,color=DarkGoldenrod]
		\tikzstyle{e:colored} = [line width=1.2pt,color=BostonUniversityRed,cap=round,opacity=0.8]
		\tikzstyle{e:coloredthin} = [line width=1.1pt,opacity=0.8]
		\tikzstyle{e:coloredborder} = [line width=2pt]
		\tikzstyle{e:main} = [line width=1pt]
		\tikzstyle{e:extra} = [line width=1.3pt,color=LavenderGray]
		
		\begin{pgfonlayer}{main}
		
		\node (C) [] {};
		
		\node (C1) [v:ghost, position=180:25mm from C] {};
		
		\node (C2) [v:ghost, position=0:0mm from C] {};
		
		\node (C3) [v:ghost, position=0:25mm from C] {};

		

		
		
		\node (v1) [v:main,position=90:18mm from C2] {};
		\node (v2) [v:main,position=180:12mm from C2] {};
		\node (v3) [v:main,position=270:18mm from C2] {};
		\node (v4) [v:main,position=0:12mm from C2] {};
		
		\node (g1) [v:ghost,position=180:4.5mm from v2] {};
		\node (g2) [v:ghost,position=0:4.5mm from v4] {};
		
		\node (g3) [v:ghost,position=180:1mm from v1] {};
		\node (g4) [v:ghost,position=0:1mm from v3] {};
		
		
		
		

		

		
		
		\draw (v2) [e:main,->,bend right=15,line width=2pt,color=CornflowerBlue] to (v4);
		\draw (v4) [e:main,->,bend right=15,densely dashed,line width=2,color=DarkGoldenrod] to (v2);
		
		\draw (v1) edge[line width=1pt,quick curve through={(g2)},->,densely dotted,line width=2pt,color=DarkMagenta] (g4);
		\draw (v3) edge[line width=1pt,quick curve through={(g1)},->,densely dotted,line width=2pt,color=DarkMagenta] (g3);
		
		\draw (v1) [e:main,->,line width=2pt,color=CornflowerBlue] to (v2);
		\draw (v2) [e:main,->,densely dashed,line width=2,color=DarkGoldenrod] to (v3);
		\draw (v3) [e:main,->,densely dashed,line width=2,color=DarkGoldenrod] to (v4);
		\draw (v4) [e:main,->,line width=2pt,color=CornflowerBlue] to (v1);
		
		

		
		
		\end{pgfonlayer}
		

		\begin{pgfonlayer}{background}
		
		\end{pgfonlayer}	
		
		\begin{pgfonlayer}{foreground}

		\end{pgfonlayer}
		\end{tikzpicture}
	\end{center}
	\caption{A closed chain of $3$ directed cycles in $A_4$.}
	\label{fig:A4chain}
\end{figure}
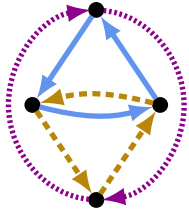

In what follows we show that, if $D$ contains a closed chain of $\ell\geq 3$ cycles, it has a haven of order $3$.
So let $\Brace{C_1,\dots, C_{\ell}}$ be a closed chain of cycles and $\ell\geq 3$.
For every $i\in\Set{1,\dots,\ell}$ let us fix a vertex $v_i\in\Fkt{V}{C_{i-1}|}\cap\Fkt{V}{C_i}$, where $C_{-1}=C_{\ell}$.
We say that a set $S\subseteq\Fkt{V}{D}$ \emph{covers} a directed cycle $C_i$ if $S\cap\Fkt{V}{C_i}\neq\emptyset$.
Next we will define the function $h$ that will be our haven.

Let $S\subseteq\Fkt{V}{D}$ with $\Abs{S}=1$.
In this case $S$ covers at most two of the directed cycles in our chain and these must be consecutive.
Hence without loss of generality we may assume $S$ covers $C_1$ and $C_{\ell}$.
In this case $\Brace{C_2,\dots,C_{\ell-1}}$ is a chain of cycles in $D-S$ and this it must be contained in a single strong component $K$ of $D-S$.
We set $\Fkt{h}{S}\coloneqq K$.
Similarly, if $S$ covers only one directed cycle of our closed chain there still is a unique strong component $K$ containing all other directed cycles and if $S$ covers no $C_i$ at all, we choose $K$ as the unique component containing the whole closed chain.
Also in these cases we set $\Fkt{h}{S}\coloneqq K$.

From now on let us assume $S\subseteq\Fkt{V}{D}$ contains exactly two vertices.
Now $S$ can cover up to four directed cycles from the closed chain.
If it covers non, one, or two consecutive ones we define $K$ in the same way as we did in the case of $\Abs{S}=1$ and set $\Fkt{h}{S}=K$.
For the remaining cases we do a case distinction.
In each of the cases we choose some $h\in\Set{1,\dots,\ell}$ and take $K$ to be the strong component of $D-S$ that contains the vertex $v_h$.
Every $S\subseteq\Fkt{V}{D}$ of size at most two can cover at most four directed cycles of our chain.
Moreover, with $\ell\geq 3$ there is always some $v_h\in\Set{v_1,\dots,v_{\ell}}\setminus S$ and we may choose $\Fkt{h}{S}$ to be the strong component of $D-S$ containing $v_h$.

What is left is to show that $h$ indeed is a haven.
Let $S=\Set{x,y}\subseteq\Fkt{V}{D}$, it suffices to show that $\Fkt{h}{\Set{x,y}}\subseteq\Fkt{h}{\Set{x}}$.
Since $x$ can cover at most two of the directed cycles of our chain, and these must be consecutive, all non-covered $C_i$ are completely contained in $\Fkt{h}{\Set{x}}$.
Moreover, since $\Fkt{h}{\Set{x,y}}$ is defined as the strong component of $D-\Set{x,y}$ containing $v_h$, $\Fkt{h}{\Set{x,y}}$ is contained in the strong component of $D-x$ containing $v_h$, which is exactly $\Fkt{h}{\Set{x}}$.
\end{proof}

With \cref{lemma:smallhavens} we know that a digraph of directed treewidth one cannot contain a bicycle or $A_4$ as a butterfly minor.
Invoking \cref{thm:strongminors} we obtain the following corollary.

\begin{corollary}\label{cor:allowedminors}
If $D$ is a digraph with $\dtw{D}=1$, then every strongly $2$-connected butterfly minor of $D$ is isomorphic to $\Bidirected{K_2}$.
\end{corollary}

A way to exploit this is the fact that every dibrace of $D$ is strongly $2$-connected and thus, if $D$ has directed treewidth one, every dibrace of $D$ has exactly $2$ vertices.
Hence, if $\Brace{T,\sigma,\zeta,\mathcal{S}}$ is an $\mathcal{S}$-decomposition for $D$, then $\Abs{B_t}=2$ for all $t\in\Fkt{V}{T}$.

\begin{lemma}\label{lemma:nogreatcycletodtw1}
Let $D$ be a strongly connected digraph on at least two vertices that does not contain a bicycle or $A_4$ as a butterfly minor, then $\dtw{D}=1$.
\end{lemma}

\begin{proof}
If $D$ has exactly two vertices there is nothing to show, so we may assume $\Abs{\Fkt{V}{D}}\geq 3$.
Let $\mathcal{S}$ be a maximal laminar family of non-trivial tight separations and $\Brace{T,\sigma,\zeta,\mathcal{S}}$ an $\mathcal{S}$-decomposition for $D$.
By \cref{cor:allowedminors} every dibrace of $D$ must be isomorphic to $\Bidirected{K_2}$ and thus $\Abs{B_t}\leq 2$ for all $t\in\Fkt{V}{T}$.
Hence $D$ must have at least one non-trivial tight separation.
Let us choose an arbitrary leaf $r\in\Fkt{V}{T}$ as a root and let $t$ be its unique neighbour in $T$.

We define a directed tree decomposition for $D$ on the tree $T$.
First let $\Set{B_r}=\Set{x,y}$, where $y$ is the cut vertex of the tight separation $\Fkt{\sigma}{rt}$, we set $\Fkt{\beta}{r}=\Set{x,y}$.
From now on we regard $T$ as a directed graph where all edges are oriented away from $r$.
Let $v\in\Fkt{V}{D}\setminus{x,y}$ be any vertex, then let $t$ be the unique vertex closest to $r$ with $v\in B_t$.
To see the uniqueness of $t$ suppose there are $t_1$ and $t_2$ both at the same distance from $r$ with $v\in B_{t_i}$.
Then there is a vertex $t'$ in $T$ that is the smallest common predecessor of the $t_i$ and there is a unique edge $e$ incident with $t'$ on the path from $t'$ to $t_1$, but no $t'$-$t_2$-path in $T$ contains $e$.
By choice of $t'$ we know $v\notin B_{t'}$.
But then $v$ must be contained in both shores of the directed separation $\Fkt{\sigma}{e}$, implying $v$ to be in the separator and thus $v\in B_{t'}$.
Now let $\Brace{t,t'}\in\Fkt{E}{T}$ and let $s$ be the unique vertex in the separator of $\Fkt{\sigma}{tt'}$.
We set $\Fkt{\gamma}{\Brace{t,t'}}\coloneqq\Set{s}$.
Then every closed directed walk starting in $\Fkt{\beta}{T_t'}$, leaving the set and returning must contain $s$ and thus $s$ is a proper guard.

Suppose there is some $t\in\Fkt{V}{T}$ such that $\Fkt{\beta}{t}=\emptyset$ and let $\Brace{t',t}\in\Fkt{E}{T}$ as well as $e$ be any outgoing edge of $t$.
Then the separators of $\Fkt{\sigma}{t't}$ and $\Fkt{\sigma}{e}$ must coincide since otherwise the separator of $\Fkt{\sigma}{e}$ would contain a vertex $s$ for which $t$ would be the first vertex, seen from $r$, with $s\in B_t$.
Every dibrace of $D$ contains at least two vertices and thus there must exist a vertex $u\in B_t$ that is not the cut vertex of $\Fkt{\sigma}{t't}$.
But in this case $u$ cannot be contained in $B_{t''}$ for any $t''\in\Fkt{V}{T}\setminus\Set{t}$ and thus $u\in\Fkt{\beta}{t}$.
Hence $\CondSet{\Fkt{\beta}{t}}{t\in\Fkt{V}{T}}$ is a partition of $\Fkt{V}{D}$ and thus $\Brace{T,\beta,\gamma}$ is a directed tree decomposition.
Moreover, $B_t=\Fkt{\Gamma}{t}$ for all $t\in\Fkt{V}{T}$ and thus $\Width{\Brace{T,\beta,\gamma}}=1$.
\end{proof}

\begin{lemma}\label{lemma:dtw1tohw1}
If $D$ is a strongly connected digraph with $\dtw{D}=1$, then $\CycleHypergraph{D}$ is a hypertree.
\end{lemma}

\begin{proof}	
We prove the claim by induction over $\Abs{\Fkt{V}{D}}$.
If $\Abs{\Fkt{V}{D}}=2$ $D$ must be isomorphic to $\Bidirected{K_2}$ and thus $\CycleHypergraph{D}$ is a hypertree.	

Now let $\Abs{\Fkt{V}{D}}\geq 3$.
By \cref{lemma:nogreatcycletodtw1} there exists a directed tree decomposition $\Brace{T,\gamma,\beta}$ of width $1$ for $D$.
Let $X=\Set{x,y}\subseteq\Fkt{V}{D}$ be the shore of a non-trivial tight separation with cut vertex $y$ in $D$ and let $D'$ be the digraph obtained from $D$ by contracting $X$ into the vertex $v_X$.
Such an $X$ must exist since all dibraces of $D$ are isomorphic to $\Bidirected{K_2}$ by \cref{cor:allowedminors}.
Clearly $D'$ is strongly connected.
To see $\dtw{D'}=1$ let $\Fkt{\Gamma}{r}=\Set{x,y}$ and replace every occurrence of $x$ or $y$ in a bag or guard of $\Brace{T,\beta,\gamma}$ by $v_X$ and then delete the vertex $r$.
The result is again a directed tree decomposition for $D'$ of width $1$.

By induction $\CycleHypergraph{D'}$ is a hypertree and thus there exists a tree $T'$ together with a family $\Brace{T'_{C'}}_{C'\in E(\CycleHypergraph{D'})}$ of subtrees of $T'$ such that $\Fkt{V}{T'_{C'}}=\Fkt{V}{C'}$ for all $C'\in\Fkt{E}{D'}$.
Moreover, $v_X\in\Fkt{V}{T'}$.
Now for every directed cycle $C$ of $D$ that contains the vertex $x$ there is a directed cycle $C'$ in $D'$ that is a result of the contraction.
If there is a directed cycle $C''$ containing $x$ in $D$ and a directed cycle $C'''$ such that $C''-x-y=C'''-y$, in $D'$ these two cycles are indistinguishable and we introduce a copy of $T_{C'''}$ to our family of subtrees.
Now for every $C$ in $D$ containing $x$ take a corresponding directed cycle $C'$ in $D'$, replace the vertex $v_X$ by $y$ and add $x$ together with the edge $xy$.
Do the same modification for $T'$.
The result is a tree $T''$ together with a family of subtrees $\Brace{T''_C}_{C\in E(\CycleHypergraph{D})}$ where $\Fkt{V}{T''_C}=\Fkt{V}{C}$ for all directed cycles $C$ in $D$.
Hence $\CycleHypergraph{D}$ is a hypertree.
\end{proof}

\begin{lemma}\label{lemma:hypertrees}
If $D$ is a digraph such that $\CycleHypergraph{D}$ is a hypertree, $\dtw{D}=1$.
\end{lemma}

\begin{proof}
With $\CycleHypergraph{D}$ being a hypertree there exists a tree $T$ on the vertices of $D$ together with a family $\Brace{T_C}_{C\in E{\CycleHypergraph{D}}}$ of subtrees of $T$ with $\Fkt{V}{T_C}=\Fkt{V}{C}$ for all directed cycles $C$.
Choose an arbitrary leaf $r\in\Fkt{V}{T}$ as a root.
Then introduce for every vertex $t\in\Fkt{V}{T}$ a bag with $\Fkt{\beta}{t}=\Set{t}$ and for every edge $\Brace{d,t}$ of $T$, where the orientation of $dt$ is chosen to point away from $r$, define $\Fkt{\gamma}{\Brace{dt}}=\Set{d}$.
If we define $\Fkt{\Gamma}{t}$ as in the definition of directed tree decompositions it is clear that $\Abs{\Fkt{\Gamma}{t}}\leq 2$ for all $t\in\Fkt{V}{T}$.
Moreover, if $C$ is a directed cycle containing a vertex of $\Fkt{\beta}{T_t}$ and a vertex of $\Fkt{V}{D}\setminus\Fkt{\beta}{T_t}$, then $C$ must contain $t$ itself and the predecessor $d$ of $t$ as well.
Hence $\Brace{T,\beta,\gamma}$ is in fact a directed tree decomposition of width $1$.
\end{proof}

\section{Concluding Remarks}

In this paper we have answered the open question of a precise characterisation of digraph of directed treewidth one in terms of forbidden butterfly minors.
Moreover, we established a close relation between the concept of directed treewidth and hypertree-width.
In particular, \cref{thm:dtw1} shows that the strongly connected digraphs $D$ with $\dtw{D}=1$ are exactly those whose cycle hypergraphs are hypertrees.
This shows that directed treewidth generalises the notion of acyclicity in the cycle hypergraph.
Some immediate questions are raised by these results.
\begin{itemize}
	
	\item Many problems are computationally hard even on digraphs of bounded directed treewidth.
	However the structure of hypertrees and hypergraphs of bounded hypertree-width admits a framework that allows for efficient algorithms.
	Are there digraphic problems that can be solved in, say XP-time, if instead of $D$ we consider the dual cycle hypergraph of $D$ as input?
	
	\item Related to the question above: In case we wanted to exploit the fact that $\hw{\Dual{\CycleHypergraph{D}}}$ is bounded by a function in $\dtw{D}$, but still consider $D$ as the input graph we run into a problem.
	Namely the size of $\Dual{\CycleHypergraph{D}}$ could already be exponential in $\Abs{\Fkt{V}{D}}$.
	
	Let us say that a class $\mathcal{D}$ of digraphs is of \emph{bounded cyclical complexity}, if there exists a constant $c$ and a computable function $f$ such that for all $D\in\mathcal{D}$ we have $\dtw{D}\leq c$ and $\Abs{\Fkt{E}{\CycleHypergraph{D}}}\leq\Abs{\Fkt{V}{D}}^{f(\dtw{D})}$.
	Are there otherwise hard problems that become tractable on classes of bounded cyclical complexity?
	
	\item Is there a nice characterisation of hypergraphs $H$ for which a digraph $D$ exists with $\CycleHypergraph{D}=H$?
	Can we recognise such hypergraphs in polynomial time?
	
\end{itemize}

Kintali \cite{kintali2013directed} introduced the notion of directed minors, generalising butterfly minors by allowing the contraction of whole directed cycles at once, in order to tackle the problem of infinite antichains for butterfly minors.
Sadly directed minors do not tie into the deep connection between digraph structure theory and matching theory in the way butterfly minors do, which is the main reason the results in this paper use the (weaker) notion of butterfly minors.
In terms of directed minors it is possible to restate \cref{thm:dtw1} to obtain the following.
\begin{corollary}
A strongly connected digraph $D$ satisfies $\dtw{D}=1$ if and only if it does not contain $\Bidirected{C_3}$ as a directed minor. 
\end{corollary}

\bibliographystyle{alphaurl}
\bibliography{literature}

\end{document}